\documentclass[12pt]{amsart}

\usepackage[utf8x]{inputenc}
\usepackage[english]{babel}
\usepackage{subfigure}
\usepackage{amsmath, amsfonts, amssymb}
\usepackage{cancel}
\usepackage{color}
\usepackage{tikz}
\usetikzlibrary{positioning,shapes,shadows,arrows}
\usepackage{enumitem}
\usepackage{rotating}
\usepackage[colorlinks=true,linkcolor=blue,citecolor=blue,urlcolor=blue]{hyperref}

\usepackage{url}
\usepackage{ccaption}
\usepackage{booktabs} 

\title[Pieri formula for Grothendieck polynomials]{Interval structure of the Pieri formula for Grothendieck polynomials}

\author{Viviane Pons}
\address{Laboratoire d'Informatique Gaspard Monge, Université Paris-Est
    Marne-la-Vallée, 5 Boulevard Descartes, Champs-sur-Marne, Marne-la-Vall\'ee cedex 2, France}


\newtheorem{Theoreme}{Theorem}[section]
\newtheorem{Corollaire}[Theoreme]{Corollary}
\newtheorem{Proposition}[Theoreme]{Proposition}
\newtheorem{Lemme}[Theoreme]{Lemma}
\newtheorem{Definition}[Theoreme]{Definition}

\newtheorem{Remarque}[Theoreme]{Remark}

\newcommand{\ltri}{\vartriangleleft}

\newcommand{\vleq}{\begin{turn}{-90}$\leq$\end{turn}}
\newcommand{\vgeq}{\begin{turn}{-90}$\geq$\end{turn}}
\newcommand{\veq}{\begin{turn}{-90}$=$\end{turn}}

\DeclareMathOperator{\Succs}{Succs}

\definecolor{Noir}{RGB}{0,0,0}
\definecolor{Rouge}{RGB}{205,35,38}
\definecolor{Bleu}{RGB}{2,60,195}
\definecolor{Vert}{RGB}{23,103,1}
\definecolor{Orange}{RGB}{255,113,15}
\definecolor{Blanc}{RGB}{255,255,255}

\tikzstyle{Coupled} = [opacity = .5, color=green, line width = 2]

\begin{document}

\maketitle

\begin{abstract}
We give a combinatorial interpretation of a Pieri formula for double Grothendieck polynomials in terms of an interval of the Bruhat order.
Another description had been given by Lenart and Postnikov in terms of chain enumerations. We use Lascoux's interpretation 
of a product of Grothendieck polynomials as a product of two kinds of generators of the 0-Hecke algebra, or sorting operators. 
In this way, we obtain a direct proof of the result of Lenart and Postnikov and then prove that the set of permutations 
occuring in the result is actually an interval of the Bruhat order. 
\end{abstract}

\section{Introduction}
\label{sec:IntroPolynomials}

Schubert calculus is an old topic in algebraic geometry. 
Originally, it involves counting lines satisfying some intersection conditions,
which amounts to finding the cardinalities of some zero-dimensional intersections
of  Schubert subvarieties in a Grassmanian. 
This is done by working in the cohomology or in the Chow ring, a quotient
of the ring of symmetric functions in which Schubert varieties are represented by Schur functions (see, e.g., \cite{Fulton}).

Modern intersection theory deals with more refined intersection conditions between chains of subspaces (flags), and
even between flag bundles over algebraic varieties \cite{FultonLascoux}. This involves computing in the cohomology,
equivariant cohomology,
$K$-theory (or Grothendieck ring) and equivariant $K$-theory of the flag manifold, whoses bases are respectively
the Schubert polynomials, double Schubert polynomials, Grothendieck polynomials and double Grothendieck polynomials,
all of which have  been introduced by Lascoux and Sch\"utzenberger \cite{LascouxKBruhat,LascouxGroth}.

Despite their geometric origin, these polynomials admit elementary definitions and are of interest for combinatorics.
In the case of the flag variety relative to $GL(n, \mathbb{C} )$, the cohomology ring and the Grothendieck ring can 
be interpreted as quotients of the ring of polynomials in $x_1, \dots, x_n$, with \emph{divided differences} $\partial_i$ 
operating in the first case and \emph{isobaric divided differences} $\pi_i$ in the second case. 
The natural bases, Schubert polynomials and Grothendieck polynomials, corresponding to the Schubert varieties,
can be defined as follows. Let us recall that the symmetric group $\mathfrak{S}_n$ is generated by the elementary transpositions $s_i = (i,i+1)$ with $1 \leq i <n$. It acts on multivariate polynomials by switching the variables ($s_i$ switches $x_i$ and $x_{i+1}$). The divided differences and isobaric divided differences are two deformations of this action: 
	\begin{align}
		\label{eq:partiali}
		\partial_i &:= (1 - s_i) \frac{1}{x_i - x_{i+1}}, \\
   		\label{eq:PiOnPolynomials}
		\pi_i &:= x_i \partial_i.
   \end{align}
where operators act on their left. It is convenient to introduce a second set of isobaric divided differences, 
\begin{equation}
   		\label{eq:HatPiOnPolynomials}
		\hat{\pi}_i := \partial_i x_{i+1} = \pi_i - 1.
\end{equation}
All families $\partial_i$, $\pi_i$ and $\hat{\pi}_i$ satisfy the braid relations,
together with quadratic relations which are

\begin{minipage}{0.30\linewidth}
\hfill
   \begin{equation}
   		\label{eq:HatPiQuad}
		\partial_i \partial_i = 0,
   \end{equation}
   \hfill
\end{minipage}
\hfill
\begin{minipage}{0.30\linewidth}
\hfill
   \begin{equation}
   		\label{eq:PiQuad}
		\pi_i  \pi_i = \pi_i,
   \end{equation}
   \hfill
\end{minipage}
\hfill
\begin{minipage}{0.30\linewidth}
\hfill
   \begin{equation}
   		\label{eq:HatPiQuad}
		\hat{\pi}_i \hat{\pi}_i = - \hat{\pi}_i.
   \end{equation}
   \hfill
\end{minipage}

The $\pi_i$ and $\hat{\pi}_i$, $i=1, \dots, n-1$, both generate the 0-Hecke algebra of $ \mathfrak{S}_n$. Although relations between them are simple, there is no general combinatorial description of products involving both types of operators. 

They are used to generate two families of Grothendieck polynomials, 
$(G_{(\sigma)})_{\sigma \in \mathfrak{S}_n}$ and $(\hat{G}_{(\sigma)})_{\sigma \in \mathfrak{S}_n}$, or of 
Key polynomials $(K_v)_{v \in \mathbb{N}^n}$ and $(\hat{K}_v)_{v \in \mathbb{N}^n}$. 
By definition, the operators $\pi_i$ (resp. $\hat{\pi}_i$) act on $G_{(\sigma)}$ or $K_v$ 
(resp. $\hat{G}_{(\sigma)}$ and $\hat{K}_v$) as sorting operators on the indices. 
More precisely, let $y_1, \dots, y_n$ be a second set of variables, 
we define \emph{dominant Grothendieck polynomials} and \emph{dominant Key polynomials} by
\begin{align}
	\label{eq:DominantGrothendieck}
	G_{(\omega)} &:= \prod_{\substack{i = 1 \ldots n\\ j = 1 \ldots n-i}} (1-y_jx_i^{-1}), \\
	\label{eq:DominantKeyPolynomials}
	K_\lambda = \hat{K}_\lambda &:= x_1^{\lambda_1}x_2^{\lambda_2}\dots x_n^{\lambda_n},
\end{align}
where $\omega = [n, n-1, \dots, 1]$ is the maximal permutation of size $n$, 
and $\lambda = (\lambda_1, \dots, \lambda_n)$ is a dominant vector,
 \textit{i.e.}, $\lambda_1\!\geq\!\lambda_2\!\geq\!\dots\!\geq\!\lambda_n$.
 One can also define $\hat{G}_{(\omega})$ but we do not need it. 
The Grothendieck polynomials are all the different images of the dominant polynomial by iterations of the operators $\pi_i$:
\begin{equation}
\label{eq:GrothendieckPerm}
G_{(\sigma s_i)} = G_{(\sigma)}\pi_i
\end{equation}
if $\sigma(i) > \sigma(i+1)$. The same definition holds for Key polynomials, $s_i$ acting on a vector $v$ rather than on a permutation. One uses the $\hat{\pi}_i$ to define the $(\hat{K}_\sigma)$ or $(\hat{G}_{(\sigma)})$ bases. 

Grothendieck polynomials form a linear basis of
the space of  multivariate polynomials,
regarded as a free module over the ring of symmetric polynomials. 
Thus, an interesting problem is to understand the product in terms of this basis. 
For symmetric functions, the product is completely determined by the Pieri formula, which describes the product of a Schur function by a complete function \cite{Macdonald}.
On the Grothendieck basis, an analogue operation would be to expand the product $G_{(\sigma)} G_{(s_k)}$ where $s_k$ is a simple transposition. 
This is indeed sufficient to completely characterize the multiplicative structure of the Grothendieck ring. Let us recall that 
\begin{equation}
\label{eq:G_sk}
G_{(s_k)} = 1 - \frac{y_1 \dots y_k}{x_1 \dots x_k}.
\end{equation}
In \cite[Theorem 6.4]{LascouxGroth}, Lascoux gives the following interpretation of this product in terms of $\pi$ and $\hat{\pi}$. 
\begin{Theoreme}{(Lascoux)}
\label{thm:PieriGroth}
Let $\sigma \in \mathfrak{S}_n$ and $1 \leq k < n$. Let  $\zeta = \zeta(\sigma,k)$ be the element of the coset $\sigma (\mathfrak{S}_k \times \mathfrak{S}_{n-k})$ with maximal length, \emph{i.e.}, $\zeta(1)\!>\!\zeta(2)\!>\dots>\!\zeta(k)$ and $\zeta(k+1)\!>\dots>\!\zeta(n)$. Then, modulo the ideal $ \mathfrak{Sym}(x) = \mathfrak{Sym}(y)$, that is identifying a symmetric function of $x$ with the same function of $y$, one has
\begin{equation}     
\label{eq:PieriGrothFest}
G_{(\sigma)} \frac{y_{\sigma_1}\cdots y_{\sigma_k} }{x_1\cdots x_k} 
\equiv G_{(\omega)}  \hat{\pi}_{\omega \zeta} 
       \pi_{\zeta^{-1}\sigma},
\end{equation}
where if $s_{i_1} \dots s_{i_m}$ is a reduced decomposition of a permutation $\mu$, then $\pi_\mu = \pi_{i_1} \dots \pi_{i_m}$ and $\hat{\pi}_\mu = \hat{\pi}_{i_1} \dots \hat{\pi}_{i_m}$.
\end{Theoreme}

This shows that the Pieri formula for Grothendieck polynomials can be calculated from the  formal expansion of a product of sorting operators. 
In \cite[Corollary 8.2]{LenartPostnikov}, Lenart and Postnikov give a combinatorial interpretation of this expansion for all types 
in terms of chains of the Bruhat order. In type $A$, it has the following description.
\begin{Theoreme}{(Lenart and Postnikov)}
\label{thm:LenartPostnikov}
Given a simple transposition $s_k$, we form the list of transpositions $(r_1, \dots r_\ell) = ( (1,n), (2,n), \dots, (k,n), \allowbreak (1,n-1), \dots, (k,n-1), \dots, (1,k+1), \dots, (k,k+1))$. Then
\begin{equation}
\label{eq:LenartPostnikov}
G_{(\sigma)}G_{(s_k)} \equiv G_{(\sigma)} - \frac{y_1 \dots y_k}{y_{\sigma_1} \dots y_{\sigma_k}} \sum_J (-1)^{\vert J \vert}  G_{(w(J))}
\end{equation}
where the sum is over subsets $J=(j_1 < j_2 \dots < j_s)$ of $(1, \dots, \ell)$ such that $\sigma \lessdot \sigma r_{j_1} \lessdot \sigma r_{j_1}r_{j_2} \lessdot \dots \lessdot \sigma r_{j_1} \dots r_{j_s} = w(J)$ is a saturated chain in the Bruhat order from $\sigma$ to $w(J)$. In other words, $J$ is a path in the Bruhat Hasse diagram between $\sigma$ and $w(J)$. 
The sum is cancellation free and the coefficients are either $1$ or $-1$ (each permutation is obtained in at most one chain).
\end{Theoreme}

This theorem actually describes the product as an enumeration of chains in the Bruhat order. Indeed, the sign of each element $G_{w(J)}$ is given by the length of $J$, or equivalently by $\ell(w(J)) - \ell(\sigma)$ where $\ell(\mu)$ is the length of the permutation $\mu$: the size of a reduced decomposition of $\mu$ is terms of simple transpositions. In this article, by a better understanding of the mixed product of $\pi$ and $\hat{\pi}$ given in Theorem \ref{thm:PieriGroth}, we prove that this chain enumeration is actually an interval of the Bruhat order. 

\begin{Theoreme}
Let us denote by $\eta(\sigma, k)$ (or simply $\eta$ if there is no ambiguity) the subset $w(J_{max})$ of maximal length 
obtained in the enumeration described in Theorem \ref{thm:LenartPostnikov}. Then
\label{thm:interval}
\begin{equation}
G_{(\sigma)} \frac{y_{\sigma_1}\cdots y_{\sigma_k} }{x_1\cdots x_k} 
\equiv \sum_{\sigma \leq \mu \leq \eta} (-1)^{\ell(\mu) - \ell(\sigma)} G_\mu.
\end{equation}
\end{Theoreme}

Theorems \ref{thm:PieriGroth} and \ref{thm:LenartPostnikov} give us that
\begin{equation}
\label{eq:direct-result}
 G_{(\omega)}  \hat{\pi}_{\omega \zeta} 
       \pi_{\zeta^{-1}\sigma} = \sum_J (-1)^{\vert J \vert}  G_{(w(J))}
\end{equation}
where the sum runs over subsets $J$ as described in Theorem \ref{thm:LenartPostnikov}. 
Our approach is to retrieve this result directly by studying the operators $\hat{\pi}$ and $\pi$ and to prove 
that the support of the sum is an interval. 
In Section \ref{sec:background}, we first recall the definition of the Bruhat order and a few properties that will be needed in the sequel. 
Then, we explain the interpretation of the $\pi$ and $\hat{\pi}$ operators in terms of the 0-Hecke algebra and of sorting operators. 
We then introduce two families of formal objects $K$ and $\hat{K}$
acted on by the $0$-Hecke algebra,
that can be interpreted either as Grothendieck polynomials or as  Key polynomials.

The first step of the proof of Theorem \ref{thm:interval} is given in Section \ref{sec:intervalClosure}, Corollary \ref{thm:intervalClosure}. After an expansion of \eqref{eq:direct-result} in the $\hat{K}$ basis and a change of basis, we prove that the support of the sum is closed by interval, \emph{i.e.} if $G_\mu$ appears in the sum, then for all $\nu$ such that $\sigma \leq \nu \leq \mu$, $G_\nu$ appers in the sum.
In Section \ref{sec:k-chains}, we study more closely the enumeration of Theorem \ref{thm:LenartPostnikov}, we give a more precise description of it and prove a few important properties. The direct proof of \eqref{eq:direct-result} is done in Section \ref{sec:k-chains-direct-proof}. 

Our main result is proved in Section \ref{sec:interval} where we show that the sum has a unique maximal element, which by Corollary \ref{thm:intervalClosure} makes it an interval. 
This is done by a direct characterization of the chains appearing in the enumeration.

In Section \ref{sec:parabolic-subgroups}, we give a generalized version of the Bruhat chain enumeration of Theorem \ref{thm:LenartPostnikov}. 
This new enumeration has no known interpretation in terms of product of polynomials but computes a more general mixed product of operators $\hat{\pi}$ and $\pi$.

\section{Background}
\label{sec:background}

\subsection{Bruhat and $k$-Bruhat order}
\label{sec:bruhat}

Theorems \ref{thm:LenartPostnikov} and \ref{thm:interval} both rely on the Bruhat order on permutations. In this Section, we recall the definition and properties of comparisons between permutations that we shall use constantly in our proofs. 
These properties can be found in \cite{LascouxCoxeter}, where the structure of the Bruhat order is investigated in detail. 
The following definition can be found in \cite{bourbaki}.
\begin{Definition} \label{def:bruhatclassic}
Let $\mu$, $\sigma$ be two elements of $\mathfrak{S}_n$. We say that $\mu$ is smaller than $\sigma$ and we write $\mu \leq \sigma$ if some subword of some reduced decomposition of $\sigma$ is a reduced decomposition of $\mu$.
\end{Definition}
Note that the Bruhat order contains both the left and right weak orders on permutations but is not their union. 
Also, the definition is symmetric, if $\mu \leq \sigma$ then $\mu^{-1} \leq \sigma^{-1}$ and it is equivalent to this local form:
\begin{Definition}
\label{def:bruhatRecurs}
$\mu$ is a successor of $\sigma$ for the Bruhat order if there is a transposition $\tau$ such that $\sigma \tau = \mu$ and  $\ell(\mu) = \ell(\sigma) + 1$. Such a transposition $\tau$ is called a \emph{Bruhat transposition of the permutation $\sigma$}.
\end{Definition}
If such a transposition exists, then there is also a transposition $\tau'$ such that $\tau' \sigma = \mu$ and the symmetry is preserved.
Ehresmann gives another criterion (historically, the first one), also described in \cite{LascouxCoxeter}.
\begin{Proposition}
\label{prop:BruhatComparisions}
Let $\mu, \sigma \in \mathfrak{S}_n$. 
\begin{align}
\label{eq:BruhatProjection}
\nonumber
\sigma \leq \mu \Leftrightarrow 
&(\forall h, 1 \leq h < n),(\forall \ell,  1 \leq \ell <n) \\
&\# \lbrace y \in p_h(\sigma_1, \dots \sigma_\ell) \rbrace \leq \# \lbrace y \in p_h(\mu_1, \dots, \mu_\ell) \rbrace
\end{align}
where $p_h$ is the projection of a permutation word on the alphabet $\lbrace x,y \rbrace$ by sending $1,\dots,h$ on $x$, and $h+1,\dots,n$ on $y$. This is equivalent to say that all reordered left factors of $\sigma$ are smaller or equal componentwise than the reordered left factors of the same size of $\mu$.
\end{Proposition}

As an example, if $\sigma = 2143$ and $\mu = 4132$, then $\sigma \leq \mu$. One has that $ \lbrace 2 \rbrace \leq \lbrace 4 \rbrace$, $ \lbrace 2,1 \rbrace \leq \lbrace 4,1 \rbrace$, $ \lbrace 4,2,1 \rbrace \leq \lbrace 4,3,1 \rbrace$. Or equivalently, all projections of $\sigma$ on the alphabet $ \lbrace x,y \rbrace$ ($yxyy$, $xxyy$, and $xxyx$ respectively) have less $y$ in their left factors than the corresponding projections of $\mu$ ($yxyy$, $yxyx$, and $yxxx$ respectively). In this Proposition, we compare left factors of the permutation and the symmetry seems broken. But actually, a more general criterion based on comparisons with bigrasmannian permutations can be given where the symmetry appears clearly. 
In this article, we shall often use a symmetry with respect to our parameter $k$, so this second criterion is  essential for us. 
It is also described in \cite{LascouxCoxeter}.

\begin{Proposition}
We have $\sigma \leq \mu$ if and only if there exists $k$ with $1 \leq k \leq n$ such that for all $1 \leq h < n$, and for all $1 \leq \ell \leq k$ and $k < \ell' \leq n$, we have
\begin{align}
\label{eq:BruhatkProjectionLeft}
\# \lbrace y \in p_h(\sigma_1, \dots \sigma_\ell) \rbrace &\leq \# \lbrace y \in p_h(\mu_1, \dots, \mu_\ell) \rbrace \text{~ and} \\
\label{eq:BruhatkProjectionRight}
\# \lbrace y \in p_h(\sigma_{\ell'}, \dots \sigma_n) \rbrace &\geq \# \lbrace y \in p_h(\mu_{\ell'}, \dots, \mu_n) \rbrace.
\end{align}
If it is true for some $k$, then it is true for all $k$ and \eqref{eq:BruhatProjection} corresponds to $k=n$.
\end{Proposition}

As an example, to obtain as above $\sigma = 2143 \leq \mu = 4132$, one can compare the first two left and right factors: $ \lbrace 2 \rbrace \leq \lbrace 4 \rbrace$, $ \lbrace 2,1 \rbrace \leq \lbrace 4,1 \rbrace$, $ \lbrace 3 \rbrace \geq \lbrace 2 \rbrace$ and $ \lbrace 4, 3 \rbrace \geq \lbrace 3, 2 \rbrace$.

We also need the notion of \emph{$k$-Bruhat order}, $\leq_k$, which was introduced in \cite{LascouxKBruhat} and studied in \cite{BergeronKbruhat}. 
\begin{Definition}
\label{def:kbruhat}
$\tau = (a,b)$ is a $k$-Bruhat transposition of $\sigma$ if $\tau$ is a Bruhat transposition of $\sigma$ by Definition \ref{def:bruhatRecurs}, and $a \leq k < b$. If $\tau$ is a $k$-Bruhat transposition of $\sigma$ and $\mu = \sigma\tau$, then $\mu$ is a successor of $\sigma$ for the $k$-Bruhat order.
\end{Definition}
There is also a characterization of the $k$-Bruhat order given in \cite[Theorem 1.1.2]{BergeronKbruhat}.
\begin{Theoreme}
\label{thm:KBruhatChar}
Let $\mu,\sigma \in \mathfrak{S}_n$, then $\sigma \leq_k \mu$ if and only if,
\begin{enumerate}[label=(\roman{*}), ref=(\roman{*})]
\item \label{cond:kbruhat1} 
$a \leq k$ implies $\sigma(a) \leq \mu(a)$, and,
$b > k$ implies $\sigma(b) \geq \mu(b)$
\item \label{cond:kbruhat3}
If $a<b$, $\sigma(a)<\sigma(b)$ and $\mu(a)>\mu(b)$ then $ a \leq k <b$.
\end{enumerate}
\end{Theoreme}

\subsection{Sorting operators, the generators of the 0-Hecke algebra}
\label{sec:IntroHecke}

The main purpose of this article is to compute the product of $\pi$ and $\hat{\pi}$ operators \eqref{eq:direct-result}. 
These operators are actually generators of the 0-Hecke algebra. 
In this Section, we give the main properties of these two families and explain how they are linked to the Bruhat order. 
Then, we interpret them as sorting operators and give a formal framework in which our problem can be expressed clearly. 

The Hecke algebra of the symmetric group $\mathfrak{S}_n$, $\mathcal{H}(t_1,t_2)$, is the algebra generated by the elements $\lbrace T_i, 1 \leq i \leq n-1 \rbrace$ and relations
\begin{align}
\label{eq:HeckeBr1}
&T_iT_j = T_jT_i, &\text{ if } |i-j|>1, \\
\label{eq:HeckeBr2}
&T_{i+1}T_iT_{i+1} = T_iT_{i+1}T_i, &\text{ if } 1  \leq i \leq n-2, \\
\label{eq:HeckeQuad}
&(T_i - t_1)(T_i - t_2)=0
\end{align} 
where $t_1$ and $t_2$ are two scalar parameters that commute with the $T_i$. The relations \eqref{eq:HeckeBr1} and \eqref{eq:HeckeBr2} are the \emph{braid relations} whereas \eqref{eq:HeckeQuad} is the quadratic relation. The algebra has a natural basis $(T_\sigma)_{\sigma \in \mathfrak{S}_n}$ where, for $s_{i_1}\ldots s_{i_m}$ a reduced decomposition of $\sigma$,

\begin{equation}
\label{eq:Tsigma}
T_\sigma = T_{i_1} \ldots T_{i_m}.
\end{equation}

The 0-Hecke algebra is the specialization of $\mathcal{H}(t_1,t_2)$ where $t_2=0$ and $t_1=1$. With these parameters, its generators $T_i$ now satisfy the quadratic relation \eqref{eq:PiQuad}, we then call them $\pi_i$ and their action on polynomials is described by \eqref{eq:PiOnPolynomials}. The family  $\pi = (\pi_\sigma)_{\sigma \in \mathfrak{S}_n}$ forms a basis of the 0-Hecke algebra. Now, if we set

\begin{equation}
\label{eq:RelPiHatPi}
\hat{\pi}_i = \pi_i - 1,
\end{equation}
these $\hat{\pi}_i$ are consistent with their definition in Section \ref{sec:IntroPolynomials} and also satisfy the braid relations \eqref{eq:HeckeBr1} and \eqref{eq:HeckeBr2} as well as the quadratic relation \eqref{eq:HeckeQuad} with specialization $t_2=0$ and $t_1=-1$, that is \eqref{eq:HatPiQuad}. Thus the family $\hat{\pi} = (\hat{\pi}_\sigma)_{\sigma \in \mathfrak{S}_n}$ constitutes a second basis of the 0-Hecke algebra generated by the $\pi_i$. The change of basis between $\pi$ and $\hat{\pi}$ is directly related to the Bruhat order on permutations.

By (\ref{eq:RelPiHatPi}), we have $\hat{\pi}_\sigma = (\pi_{i_1}-1)(\pi_{i_2}-1) \ldots (\pi_{i_m}-1)$, where $s_{i_1}s_{i_2} \ldots s_{i_m}$ is a reduced decomposition of $\sigma$, and one can prove the following property.

\begin{Proposition} \label{def:bruhatpi}
Let $\mu$, $\sigma$ be two elements of $\mathfrak{S}_n$. Then $\pi_\sigma$ (respectively $\hat{\pi}_\sigma$) can be expanded in the basis $\hat{\pi}$ (respectively $\pi$) by a sum over the Bruhat order

   \begin{align}
   		\label{eq:HatPiToPi}
		\hat{\pi}_\sigma &= \sum_{\mu \leq \sigma} (-1)^{\ell(\mu) - \ell(\sigma)} \pi_{\mu},\\
   		\label{eq:PiToHatPi}
		\pi_\sigma &= \sum_{\mu \leq \sigma} \hat{\pi}_\mu.
   \end{align}

\end{Proposition}

The fact that $\pi_\sigma$ is a sum over $\hat{\pi}_\mu$ where $\mu \leq \sigma$ is obvious by \eqref{eq:RelPiHatPi}. One has only to prove that the multiplicity of each element of the sum is $1$, \textit{i.e.}, that the non-reduced words in   $(\hat{\pi}_{i_1}+1)(\hat{\pi}_{i_2}+1) \ldots (\hat{\pi}_{i_m}+1)$ cancel. This has been done in \cite[Lemma 1.13]{LascouxGroth}.

Although the relations between $\pi$ and $\hat{\pi}$ are simply expressed by (\ref{eq:RelPiHatPi}), and both bases are well-known, the general description of mixed products of both $\pi_i$ and $\hat{\pi}_i$ is an open problem.
In this article, we compute the specific product of $\pi$ and $\hat{\pi}$ given in Section \ref{sec:IntroPolynomials}.

The 0-Hecke algebra admits a faithful realization by sorting operators. Let $K = (K_\sigma)_{\sigma \in \mathfrak{S}_n }$ and $\hat{K} = (\hat{K}_\sigma)_{ \sigma \in \mathfrak{S}_n}$, then the actions of $\pi_i$ and $\hat{\pi}_i$ on the free modules spanned by  the $K_\sigma$ and the  $\hat{K}_\sigma$ are

   \begin{equation}
   		\label{eq:PiAction}
			K_\sigma \pi_i = 
			\begin{cases}
				K_{\sigma s_i} \text{ if }\sigma_i > \sigma_{i+1}, \\
				K_\sigma \text{ otherwise},
			\end{cases}
   \end{equation}

   \begin{equation}
   		\label{eq:HatPiAction}
		\hat{K}_\sigma \hat{\pi}_i =
		\begin{cases}
			\hat{K}_{\sigma s_i} \text{ if }\sigma_i > \sigma_{i+1}, \\
			-\hat{K}_\sigma \text{ otherwise},
		\end{cases}
   \end{equation}

where $s_i$ is the simple transposition switching $\sigma_i$ and $\sigma_{i+1}$. One can check that \eqref{eq:PiAction} and \eqref{eq:HatPiAction} are consistent with the braid relations \eqref{eq:HeckeBr1} and \eqref{eq:HeckeBr2}. The case where $\sigma_i \leq \sigma_{i+1}$ is implied by the quadratic relation \eqref{eq:HeckeQuad}. 
Now, we set
\begin{equation}
\label{eq:Komega}
 K_\omega = \hat{K}_\omega
\end{equation}
where $\omega = \left[ n,n-1,\dots,1 \right]$ is the maximal permutation. Thus $K$ and $\hat{K}$ are two bases of the same module. The change of basis between $K$ and $\hat{K}$ is a direct corollary of relations (\ref{eq:HatPiToPi}) and (\ref{eq:PiToHatPi}):

   \begin{align}
   		\label{eq:HatKToK}
		\hat{K}_\sigma &= \sum_{\mu \geq \sigma} (-1)^{\ell(\mu) - \ell(\sigma)} K_\mu, \\
   		\label{eq:KToHatK}
		K_\sigma &= \sum_{\mu \geq \sigma} \hat{K}_\mu.
   \end{align}

The $K_\sigma$ and $\hat{K}_\sigma$ can be mapped to either \emph{Grothendieck polynomials} or \emph{Key polynomials} (Demazure characters for type $A$) 
by sending $K_\omega$ to \eqref{eq:DominantGrothendieck} or to \eqref{eq:DominantKeyPolynomials}. Expanding 

\begin{equation}
\label{eq:product}
K_\omega \hat{\pi}_{\omega \zeta} \pi_{\zeta^{-1} \sigma}
\end{equation}
on the $K$ basis is equivalent to expanding \eqref{eq:PieriGrothFest} on Grothendieck polynomials. The two kinds of sorting operators $\pi_i$ and $\hat{\pi}_i$ give us the flexibility of working with two bases that correspond to two different actions of the 0-Hecke algebra. 

\section{Closure by interval}
\label{sec:intervalClosure}

Mixed products of operators of both types $\pi$ and $\hat{\pi}$ are not well described in general. But we have from Theorem \ref{thm:LenartPostnikov} that the expansion of \eqref{eq:product} in the $K$ basis is a sum over a set of permutations with coefficients $1$ or $-1$ 
depending on the length of the permutation. The purpose of this section is to prove that this set is closed by interval, our first step towards
the proof of Theorem \ref{thm:interval}.

Expression \eqref{eq:product} can be expanded in both bases $K$ and $\hat{K}$. Actually,  the expansion in the $\hat{K}$ basis is easier, we give it in Proposition \ref{thm:HatK}. Then, by a change of basis, we obtain a first description of the expansion in the $K$ basis (Proposition \ref{prop:complement}). This leads directly to Corollary \ref{thm:intervalClosure}, the main result of the Section.

The first part of the computation of \eqref{eq:product} consists in applying the $\hat{\pi}_i$ operators. This is directly given by the definition of the operators and can be easily expressed on both bases $K$ and $\hat{K}$. One has

\begin{align}
\label{eq:KomegaHatPi}
K_\omega \hat{\pi}_{\omega \zeta} &= \hat{K}_\omega \hat{\pi}_{\omega \zeta} \\
\label{eq:HatKzeta}
&= \hat{K}_\zeta \\
\label{eq:BigInterval}
&= \sum_{\mu \geq \zeta} (-1)^{\ell(\mu) - \ell(\zeta)} K_\mu .
\end{align}

A reduced decomposition of $\omega \zeta = s_{i_1} s_{i_2} \dots s_{i_{\ell(\omega) - \ell(\zeta)}}$ corresponds to a path in the weak order between the maximal permutation $\omega$ and $\zeta$. It divides the values $1, \dots, n$ into two blocks: to the left or to the right of $k$. The product $\hat{\pi}_{\omega \zeta}$ is reduced and applied to the maximal permutation, and so we obtain \eqref{eq:HatKzeta}, or by a change of basis \eqref{eq:BigInterval}. Now, a reduced decomposition of $\zeta^{-1} \sigma$, used 
for the product of $\pi$ operators, is a path in the weak order between $\zeta$ and $\sigma$: it rearranges the values inside the two blocks left of $k$ and right of $k$ and does not contain any transposition $s_k$. One can apply the $\pi_{\zeta^{-1} \sigma}$ operators starting either from the point given in \eqref{eq:HatKzeta}, or from the sum over the interval $[\zeta, \omega]$ in \eqref{eq:BigInterval}. In this Section, we shall consider the first approach.

\begin{Proposition}
\label{thm:HatK}
Let $\sigma$, $k$, and $\zeta$ be as in Theorem \ref{thm:PieriGroth}, then 
\begin{equation}
\label{eq:resultHatK}
K_\omega \hat{\pi}_{\omega \zeta} \pi_{\zeta^{-1} \sigma} = \sum_{\zeta \geq \mu \geq \sigma} \hat{K}_\mu.
\end{equation}
\end{Proposition}

\begin{proof}
We start from \eqref{eq:HatKzeta} and apply the $\pi_i$ operators.

\begin{equation}
\label{eq:HatKzetaPi}
\hat{K}_\zeta \pi_{\zeta^{-1} \sigma} = \hat{K}_\zeta \sum_{\nu \leq \zeta^{-1}\sigma}\hat{\pi}_\nu.
\end{equation}
This leads directly to \eqref{eq:resultHatK}, as every subword $\nu$ of a reduced decomposition of $\zeta^{-1}\sigma$ is an element of $\mathfrak{S}_k \times \mathfrak{S}_{n-k}$. The permutation $\zeta$ is divided into two blocks, $\zeta(1)\!>\dots>\!\zeta(k)$ and $\zeta(k+1)\!>\dots>\zeta(n)$. 
Each operator sorts the permutation either in the left block or in the right block. 
\end{proof}

Note that the interval $ \left[  \sigma, \zeta \right] $ is actually the coset $\sigma ( \mathfrak{S}_k \times \mathfrak{S}_{n-k})$.
From Proposition \ref{thm:HatK}, one can obtain an expansion of \eqref{eq:product} on the $K$ basis by applying a change of basis to the sum \eqref{eq:resultHatK}.

\begin{Proposition}
\label{prop:complement}
Let $\Succs_{\sigma}^{-k}$ be the set of successors of $\sigma$, $\sigma' = \sigma \tau$ where $\tau$ is not a $k$-transposition of $\sigma$. Then (\ref{eq:product}) expanded on the $K$ basis is equal to

\begin{equation}
\label{eq:complement}
\sum (-1)^{\ell(\nu) - \ell(\sigma)} K_\nu
\end{equation}
summing over permutations $\nu$ such that $\nu \geq \sigma$ and, $\forall \sigma' \in \Succs_{\sigma}^{-k}$, $\nu \ngeq \sigma'$.
\end{Proposition}

Assuming this proposition, the fact that this set is closed by interval follows immediately.

\begin{Corollaire}
\label{thm:intervalClosure}
If $K_\nu$ appears in the sum, then $\forall \nu' \in [\sigma, \nu]$,  $K_{\nu'}$ appears in the sum.
\end{Corollaire}

To prove Proposition \ref{prop:complement}, we need the following Lemma on the Bruhat order.

\begin{Lemme}
\label{lemme:coset}
For $\sigma$, $\nu$ two permutations of $ \mathfrak{S}_n$ such that $\nu > \sigma$, and $k <n$, then the intersection of the coset $\sigma( \mathfrak{S}_k \times \mathfrak{S}_{n-k})$ and the interval $[\sigma, \nu]$ is an interval of the Bruhat order.
\end{Lemme}

\begin{proof}

Note that the Bruhat order is not a lattice and so, in general, the intersection of two intervals is not an interval : we need to prove that the intersection contains a unique maximal element.

For $\sigma \in \mathfrak{S}_n$, the coset $\sigma( \mathfrak{S}_k \times \mathfrak{S}_{n-k})$ corresponds to permutations $\mu$ with $ \lbrace \mu_1, \dots, \mu_k \rbrace = \lbrace \sigma_1, \dots, \sigma_k \rbrace$. We first describe the construction of the maximal element, $\tilde{\nu}$,  and then prove that it is indeed maximal. 

We set $V_1 = \lbrace \sigma_1, \dots, \sigma_k \rbrace$, then $\tilde{\nu}_1 = \max( v \in V_1 ~;~ v \leq \nu_1)$. Now $V_2 = V_1 \setminus \lbrace \tilde{\nu}_1 \rbrace$, and $\tilde{\nu}_2 = \max( v \in V_2 ~;~ v \leq \nu_2)$. We continue this construction up to $\tilde{\nu}_k$. For the right part of the permutation, we take $V_n = \lbrace \sigma_{k+1}, \dots, \sigma_n \rbrace$ and $\tilde{\nu}_n = \min(v \in V_n ~;~ v \geq \nu_{n})$, then $V_{n-1} = V_n \setminus \lbrace \tilde{\nu}_n \rbrace$ and we continue up to $\tilde{\nu}_{k+1}$. As an example, for $\sigma = 13245$, $k=2$ and $\nu = 54123$, we have $\tilde{\nu} = 31524$.

Firstly, as $\nu > \sigma$, we can always construct $\tilde{\nu}$. Indeed, for $i \leq k$, let $v_{min} = \min(v \in V_i)$. By construction, we have that $ \# \lbrace \sigma_j < v_{min} ~;~ 1 \leq j \leq k \rbrace = \# \lbrace \tilde{\nu}_j < v_{min} ~;~ 1 \leq j \leq i-1 \rbrace = \# \lbrace \nu_j < v_{min} ~;~ 1 \leq j \leq i-1 \rbrace$, and we also have by Bruhat that $ \# \lbrace \sigma_j < v_{min} ~;~ 1 \leq j \leq k \rbrace \geq \# \lbrace \nu_j < v_{min} ~;~ 1 \leq j \leq k \rbrace$, which tells us that $v_{min} \leq \nu_i$. A similar argument can be given for $i>k$. It is also clear that $\tilde{\nu}$ is in the coset $\sigma( \mathfrak{S}_k \times \mathfrak{S}_{n-k})$ and that $\tilde{\nu} \leq \nu$.

Next, let us take $\mu \in \sigma( \mathfrak{S}_k \times \mathfrak{S}_{n-k})$ with $\mu \leq \nu$ and prove that $\mu \leq \tilde{\nu}$. If we 
assume that  $\mu \nleq \tilde{\nu}$, this means there is $1 \leq h \leq n$ and $1 \leq i < n$ such that

\begin{equation}
\label{eq:coset-proof-1}
\# \lbrace \mu_j \geq h~;~ 1 \leq j \leq i \rbrace > \# \lbrace \tilde{\nu}_j \geq h ~;~ 1 \leq j \leq i \rbrace.
\end{equation}
We may assume that $i < k$, otherwise we can have a similar proof by working on right factors. Obviously, this means that $ \lbrace \mu_j ~;~ 1 \leq j \leq i \rbrace \neq \lbrace \tilde{\nu}_j ~;~ 1 \leq j \leq i \rbrace$ and more precisely, there is $1 \leq \ell \leq i$ such that $\mu_\ell \geq k$ and $\mu_{\ell} \in \lbrace \tilde{\nu}_j ~;~ i < j \leq k \rbrace$. We take $\mu_\ell$ minimal. Now we have:
\begin{equation}
\label{eq:coset-proof-2}
\mu_{\ell} \leq \tilde{\nu}_j \Leftrightarrow \mu_\ell \leq \nu_j
\end{equation}
for $1 \leq j \leq i$. We already had that $\mu_\ell \leq \tilde{\nu}_j$ implied $\mu_\ell \leq \nu_j$. Then, if $\mu_\ell \leq \nu_j$, as $\mu_{\ell} \in V_{j}$ and $\tilde{\nu}_j$ is maximal among $v \leq \nu_j$ in $V_j$, we have $\mu_\ell \leq \tilde{\nu}_j$. From \eqref{eq:coset-proof-1} and \eqref{eq:coset-proof-2}, we can now prove
\begin{align}
\# \lbrace \mu_j \geq \mu_{\ell} ~;~ 1 \leq j \leq i \rbrace >& \# \lbrace \tilde{\nu}_j \geq \mu_{\ell} ~;~ 1 \leq j \leq i \rbrace\\
& = \# \lbrace \nu_j \geq \mu_{\ell} ~;~ 1 \leq j \leq i \rbrace
\end{align}
which contradicts the fact that $\mu \leq \nu$. The equality is given by \eqref{eq:coset-proof-2}. Now, as we have taken $\mu_\ell$ minimal, if $\mu_j \geq h$ in \eqref{eq:coset-proof-1} and $\mu_j \notin \lbrace \tilde{\nu}_{j'} ~;~ 1 \leq j' \leq i \rbrace$, then $\mu_j \geq \mu_\ell$ which proves the first inequality. So we cannot have \eqref{eq:coset-proof-1} and $\tilde{\nu}$ is maximal.

\end{proof}

\begin{proof}[Proof of Proposition \ref{prop:complement}]
By a change of basis in the sum \eqref{eq:resultHatK}, one has

\begin{align}
\label{eq:DoubleKSum}
\sum_{\zeta \geq \mu \geq \sigma} \hat{K}_\mu &= \sum_{\zeta \geq \mu \geq \sigma} \sum_{\nu \geq \mu} (-1)^{\ell(\nu) - \ell(\mu)}K_\nu \\
\label{eq:Expansioncnu}
&= \sum_{\nu \geq \sigma} c_\nu K_\nu
\end{align}
where
\begin{equation}
\label{eq:cnu}
c_\nu = \sum_{\substack{\mu \leq \nu \\ \sigma \leq \mu \leq \zeta}} (-1)^{\ell(\nu) - \ell(\mu)}.
\end{equation}
By \cite{VermatMobius}, $c_\nu$ is a sum of values of the \emph{M\"obius function} of $\nu$ and $\mu$. Lemma \ref{lemme:coset} tells us that $c_\nu$ is always a sum over an interval. Consequently, we have $c_\nu \neq 0$ only when this interval is reduced to a single element $\sigma$, \textit{i.e.}, when there is no $\mu$ such that $\sigma < \mu \leq \zeta$ and $\mu \leq \nu$. It is sufficient to test only the direct successors of $\sigma$ in $[\sigma, \zeta]$ which correspond exactly to $Succs_{\sigma}^{-k}$.

\end{proof}

\section{Sum over chains of the $k$-Bruhat order}
\label{sec:k-chains}

In this Section, we give a more precise description of the enumeration of Theorem \ref{thm:LenartPostnikov}. We explain how the list of transpositions can be reduced to the size of the maximal chain and how the order is a linear extension of a partial order on transpositions. Then, we give a new proof of the Theorem based on the formula of Theorem \ref{thm:PieriGroth} and on the study of the mixed product of $\hat{\pi}$ and $\pi$ operators \eqref{eq:product}. We also prove some properties of the enumeration that will be used in the proof of  our main result in Section \ref{sec:interval}.

\subsection{Description}
\label{sec:k-chains-main}

\begin{Definition}
\label{def:Wsigma}
Let $\sigma \in \mathfrak{S}_n$ and $1 \leq k < n$. Let $W_{\sigma,k}$ be the list of $k$-Bruhat transpositions of $\sigma$, 
endowed with a total order on its transpositions defined by:
\begin{equation}
\label{eq:totalOrder}
(a,b) \prec (a',b') \Leftrightarrow \sigma(a) > \sigma(a') \text{ or } (a=a' \text{ and } \sigma(b) < \sigma(b'))
\end{equation}

\end{Definition}

Note that the relation $\prec$ depends on the permutation $\sigma$ and so we can only compare transpositions within a given list $W_{\sigma, k}$. We shall see that a partial order on transpositions, not depending on $\sigma$, can also be defined, and that $\prec$ is a linear extension of this partial order.

The following Proposition is an alternative formulation of Theorem \ref{thm:LenartPostnikov} as explained in Section \ref{sec:k-chains-order}.
We shall give a direct proof of it in Section \ref{sec:k-chains-direct-proof}.

\begin{Proposition}
\label{prop:ESigma}
If $W_{\sigma, k} = (\tau_1 \prec \tau_2 \prec \dots \prec \tau_m)$,  then

\begin{equation}
\label{eq:ESigmaResult}
K_\omega \hat{\pi}_{\omega \zeta} \pi_{\zeta^{-1} \sigma} = \mathfrak{E}_\sigma
\end{equation}
where
\begin{equation}
\label{eq:ESigma}
\mathfrak{E}_\sigma := K_\sigma \cdot (1 - \tau_1) \cdot (1-\tau_2) \cdots (1-\tau_m) 
\end{equation}
and
\begin{equation}
\label{eq:ProdKSigmaTau}
K_\mu \cdot \tau = \begin{cases}
K_{\mu\tau} \text{ if } \tau \text{ is a Bruhat transposition of } \mu, \\
0 \text{ otherwise.}
\end{cases}
\end{equation}
\end{Proposition}

In other words, $\mathfrak{E}_\sigma$ is a signed sum over subwords of $W_\sigma$ which are valid paths starting from $\sigma$ in the Bruhat order graph. We call such a subword simply a \textit{valid path}. For example: if $\sigma = 1362 \vert 54$, $k=4$, then $W_\sigma = ((2,6),(2,5),(4,6),(4,5))$ and 
\begin{equation}
\mathfrak{E}_{1362 \vert 54} = K_{1362 \vert 54} \cdot (1 - (2,6))\cdot(1 - (2,5))\cdot(1 - (4,6))\cdot(1 - (4,5)).
\end{equation}

When expanded, $\mathfrak{E}_\sigma$ can be represented as a tree which is a subgraph of the Bruhat order (Fig. \ref{fig:ESigma}).

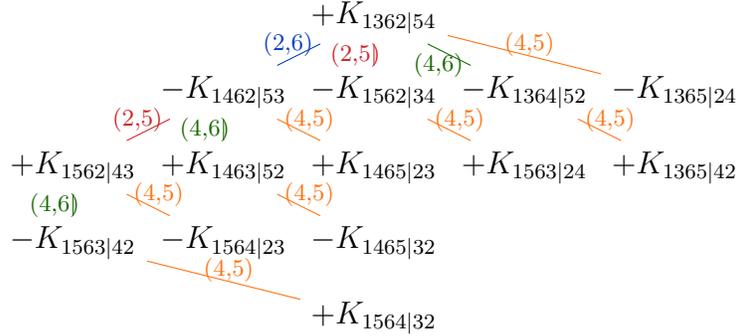
\begin{figure}[ht]
\centering
\begin{tikzpicture}
\node(K136254) at (4,0){$+K_{1362 \vert 54}$};

\node(K146253) at (2,-1){$- K_{1462 \vert 53}$};
\node(K156234) at (4,-1){$- K_{1562 \vert 34}$};
\node(K136452) at (6,-1){$- K_{1364 \vert 52}$};
\node(K136524) at (8,-1){$- K_{1365 \vert 24}$};

\node(K156243) at (0,-2){$+ K_{1562 \vert 43}$};
\node(K146352) at (2,-2){$+ K_{1463 \vert 52}$};
\node(K146523) at (4,-2){$+ K_{1465 \vert 23}$};
\node(K156324) at (6,-2){$+ K_{1563 \vert 24}$};
\node(K136542) at (8,-2){$+ K_{1365 \vert 42}$};

\node(K156342) at (0,-3){$- K_{1563 \vert 42}$};
\node(K156423) at (2,-3){$- K_{1564 \vert 23}$};
\node(K146532) at (4,-3){$- K_{1465 \vert 32}$};

\node(K156432) at (4,-4){$+ K_{1564 \vert 32}$};

\draw[color=Bleu] (K136254) -- node[ xshift=-4, yshift=4]{\scriptsize{(2,6)}} (K146253);
\draw[color=Rouge] (K136254) -- node[xshift=-7]{\scriptsize{(2,5)}} (K156234);
\draw[color=Vert] (K136254) -- node[yshift=-4, xshift=-4]{\scriptsize{(4,6)}} (K136452);
\draw[color=Orange] (K136254) -- node[yshift=4, xshift=2]{\scriptsize{(4,5)}} (K136524);

\draw[color=Rouge] (K146253) -- node[yshift=4, xshift=-4]{\scriptsize{(2,5)}} (K156243);
\draw[color=Vert] (K146253) -- node[xshift=-7]{\scriptsize{(4,6)}} (K146352);
\draw[color=Orange] (K146253) -- node[yshift=4, xshift=4]{\scriptsize{(4,5)}} (K146523);

\draw[color=Orange] (K156234) -- node[yshift=4, xshift=4]{\scriptsize{(4,5)}} (K156324);

\draw[color=Orange] (K136452) -- node[yshift=4, xshift=4]{\scriptsize{(4,5)}} (K136542);

\draw[color=Vert] (K156243) -- node[xshift=-7]{\scriptsize{(4,6)}} (K156342);
\draw[color=Orange] (K156243) -- node[yshift=4, xshift=4]{\scriptsize{(4,5)}} (K156423);

\draw[color=Orange] (K146352) -- node[yshift=4, xshift=4]{\scriptsize{(4,5)}} (K146532);

\draw[color=Orange] (K156342) -- node[yshift=4, xshift=2]{\scriptsize{(4,5)}} (K156432);

\end{tikzpicture}
\caption{The set $\mathfrak{E}_{1362 \vert 54}$.}
\label{fig:ESigma}
\end{figure}

\subsection{Partial order on $k$-transpositions}

Theorem \ref{thm:LenartPostnikov} uses a different order than Definition \ref{def:Wsigma}, but both orders are linear extensions of the following partial order.

\label{sec:k-chains-order}
\begin{Definition}
\label{def:partial-order}
We define a partial order on $k$-transpositions by
\begin{align}
\label{eq:partial-order-1}
&(a,c) \ltri (a,b) \text{ if } b<c, \\
\label{eq:partial-order-2}
&(a,c) \ltri (b,c) \text{ if } a<b.
\end{align}
\end{Definition}

\begin{Lemme}
\label{prop:linear-extension}
The order $\prec$ defined in Definition \ref{def:Wsigma} is a linear extension of $\ltri$.
\end{Lemme}

\begin{proof}
Let $W_{\sigma, k}$ be as in Definition \ref{def:Wsigma}, and $\tau, \tau' \in W_{\sigma,k}$ with $\tau \ltri \tau'$. If $\tau=(a,c)$ and $\tau'=(a,b)$ with $a<b<c$, we have that $\sigma(b)<\sigma(a)$ or $\sigma(b)>\sigma(c)$ because $\tau$ is a Bruhat transposition for $\sigma$. As $\tau'=(a,b)$ is also a Bruhat transposition for $\sigma$, then $\sigma(a)<\sigma(b)$ which makes $\sigma(b)>\sigma(c)$, so $\tau \prec \tau'$.

If $\tau=(a,c)$ and $\tau'=(b,c)$ with $a<b<c$, we still have $\sigma(b)<\sigma(a)$ or $\sigma(b)>\sigma(c)$. And as $\tau'=(b,c)$ is a Bruhat transposition for $\sigma$, then $\sigma(b)<\sigma(c)$ which makes $\sigma(b)<\sigma(a)$, so $\tau \prec \tau'$.
\end{proof}

\begin{Corollaire}
\label{cor:linear-extension}
Let $\tau, \theta \in W_{\sigma, k}$. If $\tau = (a,d)$, $\theta = (a,c)$, then $\tau \prec \theta$ implies $c<d$. And if $\tau = (a,c)$, $\theta = (b,c)$, then $\tau \prec \theta$ implies $a<b$.  
\end{Corollaire}

The proof is immediate. Lemma \ref{prop:linear-extension} and Corollary \ref{cor:linear-extension} will both contribute to the proof of
 our main result in Section \ref{sec:interval}.

The order used in Theorem \ref{thm:LenartPostnikov} is given explicitly in \cite[Theorem 4.3]{Lenart} where the combinatorial interpretation of \eqref{eq:LenartPostnikov} in terms of Bruhat chains was already proved for Grothendieck polynomials in one set of variables. It is as follows:

\begin{equation}
\label{eq:lenart-order}
(a,b) \prec' (c,d) \Leftrightarrow b>d \text{~ or ~}(b=d\text{~ and ~}a<c).
\end{equation}
It is clearly a linear extension of the partial order $\ltri$.

\begin{Proposition}
\label{prop:wsigma-generalisation}
Let $\prec'$ be any linear extension of $\ltri$. We define $W_{\sigma,k}'$ and $\mathfrak{E}_{\sigma,k}'$ just as $W_{\sigma,k}$ and $\mathfrak{E}_{\sigma,k}$ by replacing $\prec$ by $\prec'$. Then $\mathfrak{E}_{\sigma,k}' = \mathfrak{E}_{\sigma, k}$.
\end{Proposition}

\begin{proof}
Let us notice that if $\tau$ and $\theta$ are two $k$-transpositions that are not comparable by $\ltri$, they commute. So if $w' = \tau_{i_1} ' \prec' \dots \prec' \tau_{i_r}'$ is a subword of $W_{\sigma, k}'$, one can just reorder its transpositions for $\prec$ instead of $\prec'$, $w = \tau_{i_1} \prec \dots \prec \tau_{i_r}$. Now, $w$ is a subword of $W_{\sigma,k}$ and $w = w'$ when seen as a permutation. 
\end{proof}

It is now clear that if one takes a chain as in Proposition \ref{prop:ESigma} and reorder it in terms of $\prec'$, one obtains a chain as in Theorem \ref{thm:LenartPostnikov}. The other way around is not obvious as the list of transpositions described in Theorem \ref{thm:LenartPostnikov} contains all $k$-transpositions whereas the ones in $W_{\sigma, k}$ depend on the permutation $\sigma$. But the list of transpositions of Theorem \ref{thm:LenartPostnikov} can actually be reduced to $W_{\sigma, k}$.

\begin{Lemme}
\label{lemme:equiv-lenart}
All transpositions appearing in the chains of Theorem \ref{thm:LenartPostnikov} are Bruhat transpositions for the permutation $\sigma$.
\end{Lemme}

\begin{proof}
Let $\tau_1 \prec' \dots \prec' \tau_\ell$ be a chain of Theorem \ref{thm:LenartPostnikov}. Let us suppose that $\tau_i = (a,b)$ is not a Bruhat transposition for $\sigma$. It means that there exists $c$, $a < c < b$ with $\sigma(a) < \sigma(c) < \sigma(b)$. This order is not changed by the transpositions $\tau_1, \dots, \tau_{i-1}$. For $1 \leq j < i$, we set $\sigma' = \sigma \tau_1 \dots \tau_{j-1}$ and we suppose that we still have $\sigma'(a) < \sigma'(c) < \sigma'(b)$. If $\tau_j = (a,y)$, then $b<y$ so $a<c<y$ and $\sigma'(c) > \sigma'(y)$. If $\tau_j = (c,y)$, then $c<b<y$ and $\sigma'(b)>\sigma'(y)$. We cannot have $\tau_j = (y,c)$ because $\tau_j \prec' \tau_i$. If $\tau_j = (y,b)$, then $y<a$ so $y<c<b$ and $\sigma'(c)<\sigma'(y)$. By this, we have that $\tau_i$ is not a Bruhat transposition for $\sigma \tau_1 \dots \tau_{i-1}$ which contradicts the definition of the chain.
\end{proof}

The fact that each permutation is obtained by exactly one chain is given by \cite[Proposition 4.8]{Lenart}. In \cite[Algorithm 3.1.1]{BergeronKbruhat}, Bergeron and Sottile describe an algorithm that give a a saturated chain between two permutations in the $k$-Bruhat order. It is compatible with the partial order on transpositions we defined and can be used to obtain the unique chain of transpositions corresponding to each permutation. Another algorithm is also described in \cite[Algorithm 4.9]{LenartGreedyAlgo} that could be used in that purpose. Both algorithms should allow for a direct proof of the property that each permutation in the interval $\left[ \sigma, \eta \right]$ is associated with a chain of the Lenart-Postinikov enumeration of Theorem \ref{thm:LenartPostnikov}. We have shown this property algebraically in Corollary~\ref{thm:intervalClosure}. It is not enough to prove Theorem \ref{thm:interval} as one also needs to prove that each permutation of the sum is contained in the interval. This will be done in Section \ref{sec:interval}.

\subsection{Direct proof}
\label{sec:k-chains-direct-proof}

We can prove Proposition \ref{prop:ESigma} by applying the $\pi_i$ operators to the sum over an interval given in \eqref{eq:BigInterval}. First let us see:

\begin{Proposition}
\label{prop:EZeta}
\begin{equation}
\label{eq:EZeta}
\mathfrak{E}_{\zeta,k} = \sum_{\mu \geq \zeta} (-1)^{\ell(\mu) - \ell(\zeta)}K_\mu.
\end{equation}
\end{Proposition}

\begin{proof}
We know that the sign of each element $K_\mu$ of the sum $\mathfrak{E}_{\zeta,k}$ is given by $(-1)^{\ell(\mu) - \ell(\zeta)}$. So we can consider $\mathfrak{E}_{\zeta,k}$ as a set and the definition gives us that $\mathfrak{E}_{\zeta,k} \subset \left[ \zeta, \omega \right] $. We only have to prove the reverse inclusion $\left[ \zeta, \omega \right] \subset \mathfrak{E}_{\zeta,k} $.

Let $\mu \in \left[ \zeta, \omega \right]$. We first prove that $\mu \geq_k \zeta$. Proposition \ref{prop:BruhatComparisions} implies that $\mu$ and $\zeta$ satisfy Condition \ref{cond:kbruhat1} of Theorem \ref{thm:KBruhatChar}: for $a \leq k$, the left factor of size $a$ of $\mu$ is greater than the left factor of size $a$ of $\zeta$ and $\zeta(1,\dots,k)$ is antidominant, so $\zeta(a) \leq \mu(a)$. For $b > k$,  one uses that the reordered right factor of size $n-b$ of $\mu$ is smaller than, or equal to,
componentwise,  the reordered right factor of size $n-b$ of $\zeta$. Condition \ref{cond:kbruhat3} of Theorem \ref{thm:KBruhatChar} is also satisfied as for $a < b$, $\zeta(a) < \zeta(b)$ implies that $a \leq k <b$. Consequently, we have $\mu \geq_k \zeta$.

By using \cite[Algorithm 3.1.1]{BergeronKbruhat}, one obtains a chain of the $k$-Bruhat order between $\mu$ and $\zeta$. The order on the transpositions is the one described in Definition \ref{def:Wsigma}, and so if each transposition given by the algorithm is in $W_{\zeta,k}$, we obtain a subword. The process is recursive, it starts from $\mu$ and chooses a transposition $(a,b)$ such that, $a \leq k <b$, $\mu(a) > \mu(b)$ and $\zeta(a) < \zeta(b)$. As we have $\zeta(i) < \zeta(a)$ for $a <i \leq k$, and $\zeta(i) > \zeta(b)$ for $k<i<b$, then $(a,b)$ is a $k$-Bruhat transposition of $\zeta$ and is in $W_{\zeta,k}$.
\end{proof}

From \eqref{eq:BigInterval} and Proposition \ref{prop:EZeta}, we now have 

\begin{equation}
\label{eq:EZetaPi}
K_\omega \hat{\pi}_{\omega \zeta} \pi_{\zeta^{-1} \sigma} = \mathfrak{E}_{\zeta,k} \pi_{\zeta^{-1} \sigma}.
\end{equation}

We can prove Proposition \ref{prop:ESigma} by induction on the length of the product $\pi_{\zeta^{-1} \sigma}$.

\begin{Proposition}
\label{prop:InductiveStep}
Let $\sigma \in \mathfrak{S}_n$, and let us assume that \eqref{eq:ESigmaResult} holds. Let $s_i$ be a simple transposition such that $i \neq k$ and $\sigma s_i < \sigma$, then

\begin{equation}
\label{eq:InductiveStep}
\mathfrak{E}_{\sigma,k} \pi_i = \mathfrak{E}_{\sigma s_i,k}.
\end{equation}

\end{Proposition}

First, let us note the fact:

\begin{Lemme}
$\mathfrak{E}_{\sigma s_i, k} \cap \mathfrak{E}_{\sigma,k} = \emptyset$, and more precisely $ \forall \nu \in \mathfrak{E}_{\sigma s_i, k}$, $\nu \ngtr \sigma$.
\end{Lemme}
\begin{proof}
If $i \leq k$, let us consider $\vert p_{\sigma(i)-1}(\nu_1, \dots, \nu_i) \vert_y$ for $\nu \in \mathfrak{E}_{\sigma s_i, k}$, \textit{i.e.}, the number of values greater than or equal to $\sigma(i)$ in left factors of size $i$. This number is constant on $\mathfrak{E}_{\sigma s_i, k}$ and is equal to $\vert p_{\sigma(i)-1}(\sigma(1), \dots \sigma(i)) \vert_y\!-\!1$ as $\sigma s_i (i) = \sigma(i+1) < \sigma (i)$. And so, by Proposition \ref{prop:BruhatComparisions}, $\nu$ cannot be greater than $\sigma$. The same argument applies when $i>k$, by considering the number of values smaller than or equal to $\sigma(i+1)$ in right factors of size $n-i$. 
\end{proof}

Proposition \ref{prop:InductiveStep} is a consequence of these two lemmas:

\begin{Lemme}
\label{lemme:ESigmaPiNonZeroImage}
We have the following implications:
\begin{enumerate}
\item \label{lemme:ESigmaPiNonZeroImage-LtoR}
Let $w = \tau_{i_1} \dots \tau_{i_r}$ be a subword of $W_{\sigma, k}$ such that $w$ is a valid path starting from $\sigma$ in the Bruhat graph, and $\mu s_i \ngtr \sigma$ where $\mu = \sigma w$. Then, 
$$w' = (s_i\tau_{i_1}s_i)(s_i\tau_{i_2}s_i)\dots (s_i\tau_{i_r}s_i)$$
is a subword of $W_{\sigma s_i, k}$ and is a valid path in the Bruhat graph starting from $\sigma s_i$. 
\item \label{lemme:ESigmaPiNonZeroImage-RtoL}
Conversely, if $w'=t_1 \dots t_r$ is a subword of $W_{\sigma s_i, k}$ and a valid path starting from $\sigma s_i$ in the Bruhat graph, then 
$$w = (s_i t_1 s_i) \dots (s_i t_r s_i)$$ 
is a subword of $W_{\sigma, k}$ and a valid path in the Bruhat graph starting from $\sigma$.
\end{enumerate}
\end{Lemme}

\begin{Lemme}
\label{lemme:ESigmaPiZeroImage}
Let $\mu \in \mathfrak{E}_{\sigma, k}$ with $\mu s_i > \sigma$, then $\mu s_i \in \mathfrak{E}_{\sigma,k}$.
\end{Lemme}

\begin{proof}[Proof of Lemma \ref{lemme:ESigmaPiNonZeroImage}]
Implication (\ref{lemme:ESigmaPiNonZeroImage-RtoL}) is trivial: it is immediate to see that if a transposition $t$ is a $k$-Bruhat transposition for $\sigma s_i$, then $s_i t s_i$ is a $k$-Bruhat transposition for $\sigma$. Besides, the order on transpositions is not changed, and so if $t_1 \dots t_r$ is a subword of $W_{\sigma s_i, k}$ then $(s_i t_1 s_i) \dots (s_i t_r s_i)$ is a subword $W_{\sigma, k}$.

Now, let $w = \tau_{i_1} \dots \tau_{i_r}$ be a valid subword of $W_{\sigma, k}$ with $\mu = \sigma w$ such that $\mu s_i \ngtr \sigma$. For the purpose of the proof, we suppose that $i<k$. We mostly use Proposition \ref{prop:BruhatComparisions}, the same proof can be made for $i>k$ by comparing right factors instead of left factors.

We first prove that for all permutations $\tilde{\mu}$ that appear in the chain between $\sigma$ and $\mu$ given by $w$, we have $\tilde{\mu}s_i \ngtr \sigma$. As $\mu s_i \ngtr \sigma$, there is at least one left factor of $\mu s_i$ that is not greater than the corresponding factor of $\sigma$. However, all left factors of $\mu$ are greater than the factors of $\sigma$. The only factor that is different between $\mu$ and $\mu s_i$ is $\lbrace 1, \dots, i \rbrace$ and so we have

\begin{align}
&\mu s_i (1, \dots, i) \ngeq \sigma(1, \dots, i)\text{~ with }\\
&\mu s_i (1, \dots, i) = \lbrace \mu(1), \dots, \mu(i-1), \mu(i+1) \rbrace.
\end{align}
As $\tilde{\mu} <_k \mu$, we know that $\tilde{\mu}(j) \leq \mu(j)$ for all $j\leq k$ and so $\tilde{\mu} s_i (1, \dots, i) \ngeq \sigma(1, \dots, i)$. 

We can now prove that for all transpositions $\tau$ of $w$, we have $\sigma \tau (i) > \sigma \tau (i+1)$. We know that $\sigma(i) > \sigma(i+1)$ and $i<k$, so the only transposition we have to consider is $\tau = (i+1,b)$. The transposition $\tau$ is in $w$, let us say between two permutations $\nu$ and $\nu'$. We know that $\nu's_i \ngtr \sigma$ which gives $ \lbrace \nu'(1), \dots, \nu'(i-1), \nu'(i+1) \rbrace \ngeq \lbrace \sigma(1), \dots, \sigma(i) \rbrace$. As $\nu' >_k \sigma$, we have $\nu'(j) \geq \sigma(j)$ for $j \leq i-1$ and necessarily, $\nu(b) = \nu'(i+1) < \sigma(i)$.

Let us suppose that $\sigma(b) > \sigma(i)$. Then, there is a transposition $(c,b)$ in $w$ such that $\sigma(c) < \sigma(i) < \sigma(b)$ which implies $c>i$. But as $(c,b) \prec (i+1,b)$, we have $\sigma(i+1) < \sigma(c) < \sigma(b)$ which implies $c<i+1$ and leads to a contradiction. So we have $\sigma(b) < \sigma(i)$, \emph{i.e.}, $\sigma\tau(i) > \sigma \tau (i+1)$. 

It is now easy to check that for each transposition $\tau$ of $w$, $s_i \tau s_i$ is a valid Bruhat transposition for $\sigma s_i$ and so $(s_i \tau_{i_1} s_i) \dots (s_i \tau_{i_r} s_i)$ is a subword of $W_{\sigma s_i, k}$. The fact that this is a valid path comes from the already given property that if $\nu$ and $\nu' = \nu \tau$ are two permutations of the chain, we have $\nu'(i+1) < \nu(i)$. This makes $s_i \tau s_i$ a valid Bruhat transposition for $\nu s_i$.

\end{proof}

\begin{proof}[Proof of Lemma \ref{lemme:ESigmaPiZeroImage}]
We use Proposition \ref{prop:complement} which is true on $\mathfrak{E}_{\sigma, k}$ by our induction hypothesis. We show that if $\mu \in \mathfrak{E}_{\sigma, k}$, there is no $\sigma' \in Succs_\sigma^{-k}$ such that $\sigma' > \mu s_i$, which means $\mu s_i \in \mathfrak{E}_\sigma$. This is trivial for $\mu > \mu s_i$. 

Let us suppose $\mu < \mu s_i$ and $\sigma' \in Succs_\sigma^{-k}$ such that $\sigma'> \mu s_i$. As for Lemma \ref{lemme:ESigmaPiNonZeroImage}, we set $i<k$ without loss of generality. We have $\sigma' = \sigma \tau$ with $\tau = (a,b)$ a non-$k$ Bruhat transposition. 

Firstly we show that we necessarily have $a \leq i < b$. All left factors of $\mu s_i$ are greater than the left factors of $\sigma'$. There is at least one left factor of $\mu$ that is not greater than the same one of $\sigma'$ and the only one possible is $\lbrace 1, \dots, i \rbrace$. Besides, we have $\mu(1, \dots, i) \geq \sigma(1, \dots, i)$ and so $\sigma(1, \dots, i) \neq \sigma'(1, \dots, i)$.

Now we have:

\begin{equation}
\label{eq:proof-ESigmaPiZeroImage-eg1}
 \lbrace j \leq a~;~\sigma(j) < \sigma(b) \rbrace = \lbrace j \leq a~;~\mu(j) < \sigma(b) \rbrace.
\end{equation}

As $\mu(j) \geq \sigma(j)$, it is clear that we have an inclusion. If it is a strict one, then there is $c \leq a$ with $\sigma(c) < \sigma(b)$ and $\mu(c) \geq \sigma(b)$. This means that a transposition $(c,d)$ has been applied with $c < b < d$ and $\sigma(c) < \sigma(b) < \sigma(d)$ which is not possible. 

We also have:

\begin{equation}
\label{eq:proof-ESigmaPiZeroImage-eg2}
 \lbrace j \leq a~;~\mu(j) < \sigma(b) \rbrace = \lbrace j \leq a~;~\mu s_i(j) < \sigma(b) \rbrace.
\end{equation}

If $a<i$, we have $ \lbrace j \leq a~;~ \mu(j) \rbrace = \lbrace j \leq a~;~ \mu s_i(j) \rbrace$ and \eqref{eq:proof-ESigmaPiZeroImage-eg2} is true. If $a=i$, we have $\sigma(i+1) < \sigma(i) = \sigma(a) < \sigma(b)$ so $\mu(i) < \sigma(b)$ because of \eqref{eq:proof-ESigmaPiZeroImage-eg1}. And, as $b > i+1$, the argument used to prove \eqref{eq:proof-ESigmaPiZeroImage-eg1} says that $\mu s_i (i) = \mu (i+1) < \sigma(b)$ which proves \eqref{eq:proof-ESigmaPiZeroImage-eg2}. So now we have:
\begin{align}
\# \lbrace j \leq a~;~\mu s_i(j) < \sigma(b) \rbrace &= \#  \lbrace j \leq a~;~\sigma(j) < \sigma(b) \rbrace \\
&= \#  \lbrace j \leq a~;~\sigma'(j) < \sigma(b) \rbrace + 1,
\end{align}
which by Proposition \ref{prop:BruhatComparisions} says that $\mu s_i \ngtr \sigma'$ and leads to a contradiction.

\end{proof}

\begin{proof}[Proof of Proposition \ref{prop:InductiveStep}]
One consequence of Lemma \ref{lemme:ESigmaPiNonZeroImage} is that 

\begin{equation}
\mathfrak{E}_{\sigma s_i, k} = \left( \sum_{ \substack{\mu \in \mathfrak{E}_{\sigma,k} \\ \mu s_i \ngtr \sigma}} (-1)^{\ell(\mu) - \ell(\sigma)} K_\mu \right) \pi_i. 
\end{equation}

Indeed, all elements $K_\mu$ in the above sum are sorted by the operator $\pi_i$ because $\mu s_i \ngtr \sigma$ implies that $\mu s_i < \mu$. Implication 
(\ref{lemme:ESigmaPiNonZeroImage-LtoR})
 of Lemma \ref{lemme:ESigmaPiNonZeroImage} tells us they are sent on elements of $\mathfrak{E}_{\sigma s_i, k}$ and implication 
(\ref{lemme:ESigmaPiNonZeroImage-RtoL}) tells us that they cover the whole sum. Now, we have:

\begin{equation}
\mathfrak{E}_{\sigma,k} \pi_i = \mathfrak{E}_{\sigma s_i, k} + \sum_{\substack{\mu \in \mathfrak{E}_{\sigma,k} \\ \mu s_i > \sigma}} (-1)^{ \ell(\mu) - \ell(\sigma)} K_\mu \pi_i.
\end{equation} 
By Lemma \ref{lemme:ESigmaPiZeroImage}, the second part of the sum is equal to 0, as for each element $K_\mu$, its paired element $K_{\mu s_i}$ also appears in the sum with an opposite sign. Let us assume that $\mu s_i < \mu$, then by definition of $\pi_i$, $K_\mu \pi_i = K_{\mu s_i}$ and $K_{\mu s_i}\pi_i = K_{\mu s_i}$.
\end{proof}

To illustrate this proof, one can draw the following example (Fig.~\ref{fig:ESigmaPi_i}). Each elements of $\mathfrak{E}_{\sigma, k}$ (on the left) is either paired by $s_i$ to an element of $\mathfrak{E}_{\sigma s_i, k}$ (on the right) or to another element of $\mathfrak{E}_{\sigma, k}$.

\begin{figure}[ht]
\centering
\begin{tikzpicture}

\node(K136245) at (11,1) {$+K_{1362 \vert 45}$};

\node(K146235) at (10,0) {$-K_{1462 \vert 35}$};
\node(K136425) at (12,0) {$-K_{1364 \vert 25}$};

\node(K146325) at (11,-1) {$+K_{1463 \vert 25}$};

\node(K136254) at (4,0){$+K_{1362 \vert 54}$};

\node(K146253) at (2,-1){$- K_{1462 \vert 53}$};
\node(K156234) at (4,-1){$- K_{1562 \vert 34}$};
\node(K136452) at (6,-1){$- K_{1364 \vert 52}$};
\node(K136524) at (8,-1){$- K_{1365 \vert 24}$};

\node(K156243) at (0,-2){$+ K_{1562 \vert 43}$};
\node(K146352) at (2,-2){$+ K_{1463 \vert 52}$};
\node(K146523) at (4,-2){$+ K_{1465 \vert 23}$};
\node(K156324) at (6,-2){$+ K_{1563 \vert 24}$};
\node(K136542) at (8,-2){$+ K_{1365 \vert 42}$};

\node(K156342) at (0,-3){$- K_{1563 \vert 42}$};
\node(K156423) at (2,-3){$- K_{1564 \vert 23}$};
\node(K146532) at (4,-3){$- K_{1465 \vert 32}$};

\node(K156432) at (4,-4){$+ K_{1564 \vert 32}$};

\draw[color=Bleu] (K136254) -- (K146253);
\draw[color=Rouge] (K136254) --  (K156234);
\draw[color=Vert] (K136254) --  (K136452);
\draw[color=Orange] (K136254) --  (K136524);

\draw[color=Rouge] (K146253) --  (K156243);
\draw[color=Vert] (K146253) --  (K146352);
\draw[color=Orange] (K146253) --  (K146523);

\draw[color=Orange] (K156234) --  (K156324);

\draw[color=Orange] (K136452) --  (K136542);

\draw[color=Vert] (K156243) --  (K156342);
\draw[color=Orange] (K156243) --  (K156423);

\draw[color=Orange] (K146352) --  (K146532);

\draw[color=Orange] (K156342) --  (K156432);

\draw[color=Bleu] (K136245) -- (K146235);
\draw[color=Vert] (K136245) -- (K136425);
\draw[color=Vert] (K146235) -- (K146325);

\draw[Coupled] (K136254) -- (K136245);
\draw[Coupled] (K146253) -- (K146235);
\draw[Coupled] (K136452) -- (K136425);
\draw[Coupled] (K146352) -- (K146325);

\draw[Coupled] (K156234) -- (K156243);
\draw[Coupled] (K136524) -- (K136542);
\draw[Coupled] (K146523) -- (K146532);
\draw[Coupled] (K156324) -- (K156342);
\draw[Coupled] (K156423) -- (K156432);

\end{tikzpicture}
\caption{$\mathfrak{E}_{1362 \vert 54} \pi_5 = \mathfrak{E}_{1362 \vert 45}$}
\label{fig:ESigmaPi_i}
\end{figure}
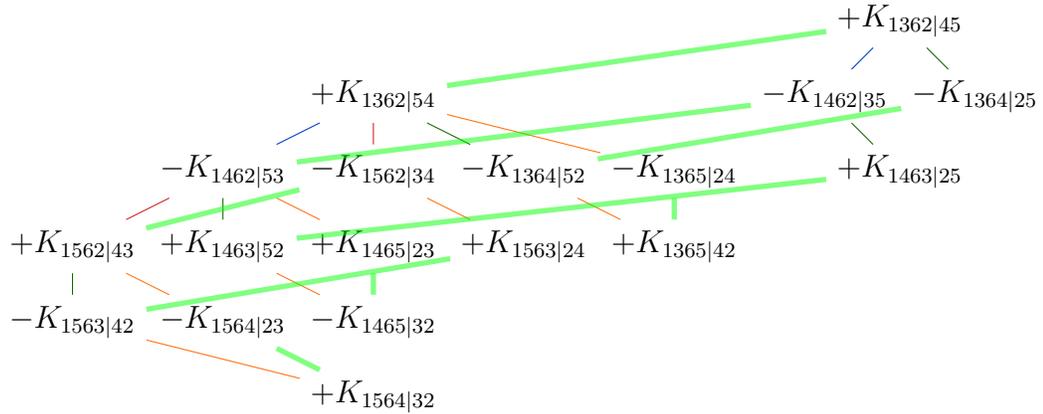

\begin{figure}[ht]
\centering
\begin{tikzpicture}

\node(K136245) at (10,1) {$+K_{1362 \vert 45}$};

\node(K146235) at (9,0) {$-K_{1462 \vert 35}$};
\node(K136425) at (11,0) {$-K_{1364 \vert 25}$};

\node(K146325) at (10,-1) {$+K_{1463 \vert 25}$};

\node(K136254) at (4,0){$+K_{1362 \vert 54}$};

\node(K146253) at (2,-1){$- K_{1462 \vert 53}$};
\node(K156234) at (4,-1){$- K_{1562 \vert 34}$};
\node(K136452) at (6,-1){$- K_{1364 \vert 52}$};
\node(K136524) at (8,-1){$- K_{1365 \vert 24}$};

\node(K156243) at (0,-2){$+ K_{1562 \vert 43}$};
\node(K146352) at (2,-2){$+ K_{1463 \vert 52}$};
\node(K146523) at (4,-2){$+ K_{1465 \vert 23}$};
\node(K156324) at (6,-2){$+ K_{1563 \vert 24}$};
\node(K136542) at (8,-2){$+ K_{1365 \vert 42}$};

\node(K156342) at (0,-3){$- K_{1563 \vert 42}$};
\node(K156423) at (2,-3){$- K_{1564 \vert 23}$};
\node(K146532) at (4,-3){$- K_{1465 \vert 32}$};

\node(K156432) at (4,-4){$+ K_{1564 \vert 32}$};

\draw[color=Bleu, line width = 2] (K136254) -- (K146253);
\draw[color=Rouge] (K136254) -- node{$\times$}  (K156234);
\draw[color=Vert, line width = 2] (K136254) --  (K136452);
\draw[color=Orange] (K136254) -- node{$\times$} (K136524);

\draw[color=Rouge] (K146253) -- node{$\times$}  (K156243);
\draw[color=Vert, line width = 2] (K146253) --  (K146352);
\draw[color=Orange] (K146253) -- node{$\times$} (K146523);

\draw[color=Orange] (K156234) -- node{$\times$} (K156324);

\draw[color=Orange] (K136452) -- node{$\times$} (K136542);

\draw[color=Vert] (K156243) --  (K156342);
\draw[color=Orange] (K156243) -- node{$\times$} (K156423);

\draw[color=Orange] (K146352) -- node{$\times$} (K146532);

\draw[color=Orange] (K156342) -- node{$\times$} (K156432);

\draw[color=Bleu, line width = 2] (K136245) -- (K146235);
\draw[color=Vert, line width = 2] (K136245) -- (K136425);
\draw[color=Vert, line width = 2] (K146235) -- (K146325);

\end{tikzpicture}
\caption{Cutting branches}
\label{fig:ESigmaCutTree}
\end{figure}
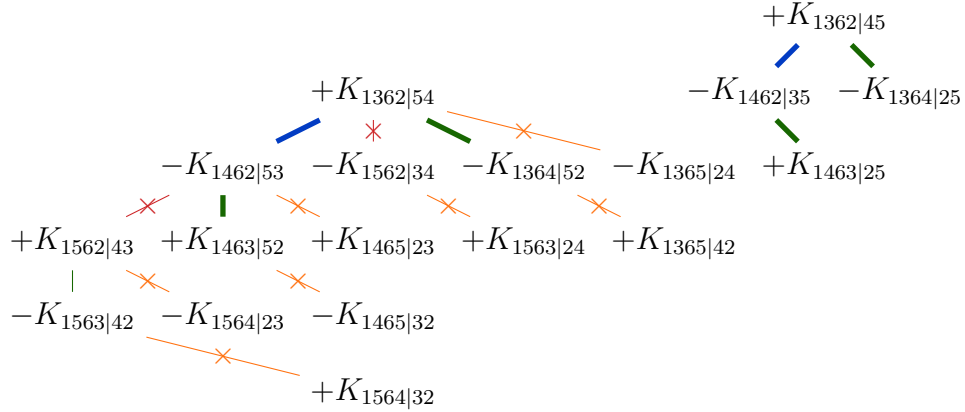

Also, $\mathfrak{E}_{\sigma s_i, k}$ is actually a conjugate subtree of $\mathfrak{E}_{\sigma, k}$. The transpositions of $W_{\sigma, k} = (\tau_1, \dots \tau_m)$ are the root branches of the tree. Proposition \ref{lemme:ESigmaPiNonZeroImage} tells us that $W_{\sigma s_i, k}$ is a conjugate subword of $W_{\sigma, k}$: only transpositions $\tau$ where $s_i\tau s_i$ is still a Bruhat transposition of $\sigma s_i$ are kept. The operator $\pi_i$ "cuts" branches out of the $\mathfrak{E}_\sigma$ tree and reduces its size exponentially (Fig. \ref{fig:ESigmaCutTree}).

\section{Interval structure of $\mathfrak{E}_{\sigma, k}$}
\label{sec:interval}

By Corollary \ref{thm:intervalClosure}, we have that $\mathfrak{E}_{\sigma, k}$ is closed by interval, \emph{i.e.}, if $ \mu, \nu \in \mathfrak{E}_{\sigma, k}$ with $\nu < \mu$, then $[\nu, \mu] \subset \mathfrak{E}_{\sigma, k}$. To prove that $ \mathfrak{E}_{\sigma, k}$ is an interval, we now have to show that there is a unique minimal element and a unique maximal element. The minimal element is by definition $\sigma$ and it is unique. The maximal element can only be $\sigma W_{\sigma, k}$, \emph{i.e.}, the permutation where all transpositions of $W_{\sigma, k}$ have been applied. It is the only element with maximal length $\ell(\sigma) + \vert W_{\sigma, k} \vert$. So the proof of Theorem \ref{thm:interval} will be completed if we prove the following Lemma:

\begin{Lemme}
\label{lemme:intervalMax}
\begin{equation}
\label{eq:intervalMax}
\forall \mu \in \mathfrak{E}_{\sigma, k},~~ \mu \leq \eta(\sigma, k)
\end{equation}
where $\eta(\sigma, k) = \sigma W_{\sigma, k}$.
\end{Lemme}

To prove Lemma \ref{lemme:intervalMax}, we need to study more precisely the structure of $\mathfrak{E}_{\sigma, k}$. 

\begin{Definition}
\label{def:compatible}
Seeing $W_{\sigma, k} = (\tau_1 \prec \dots \prec \tau_m)$ as a word on transpositions, let $w = \tau_{i_1} \dots \tau_{i_r}$ be a subword of $W_{\sigma, k}$. We say that $w$ is not compatible if:
\begin{enumerate}
\item $w$ contains $(a,c)(b,d)$ as a subword with $a<b<c<d$,
\label{enum:compatible1}
\item $(a,d) \in W_{\sigma, k}$ and $(a,d) \notin  w$.
\label{enum:compatible2}
\end{enumerate}
Otherwise, $w$ is called compatible.
\end{Definition}

Note that if $w$ contains a subword $(a,c) \prec (b,d)$ as in condition (\ref{enum:compatible1}), then if there is a transposition$(a,d) \in W_{\sigma, k}$, we always have $(a,d) \prec (a,c) \prec (b,d)$ because of the order on transpositions (see Lemma \ref{prop:linear-extension}).

In Theorem \ref{thm:LenartPostnikov} as well as in Proposition \ref{prop:ESigma}, the enumeration is only given in term of the Bruhat order. Now we state our key proposition which gives a direct, Bruhat-order free, description of the enumeration.

\begin{Proposition}
\label{prop:validSubwords}
If $w$ is a subword of $W_{\sigma, k}$, $w$ is compatible if and only if $w$ is a valid path in the Bruhat order starting at $\sigma$, \emph{i.e.}, $\sigma w \in \mathfrak{E}_{\sigma, k}$.
\end{Proposition}

Definition \ref{def:compatible} and Proposition \ref{prop:validSubwords} rely on the order on transpositions of $W_{\sigma, k}$ given in Definition \ref{def:Wsigma}. We use the total order for convenience and clarity of the proof. But the crucial arguments come from the partial order defined in Definition \ref{def:partial-order} and a similar description of non compatible subwords could be given for any other linear extension of the partial order. 

We first prove the following lemma:

\begin{Lemme}
\label{lemme:non-inversion}
Let $w = \tau_{i_1} \dots \tau_{i_r}$ be a subword of $W_{\sigma, k}$ which is a valid path, and $\tau = (a,b) \in W_{\sigma, k}$, $\tau \succ \tau_{i_r}$. Then, 
\begin{equation}
\sigma w (a) < \sigma w (b)
\end{equation}
\end{Lemme}

\begin{proof}
Note that $\sigma(a) < \sigma(b)$ by definition of $\tau$, and as $\tau \succ \tau_{i_r}$, it has not yet been applied. Let $\theta \in w$ be a transposition with $w = w' \theta w''$. We set $\sigma' = \sigma w'$ and we suppose that

\begin{equation}
\label{eq:proof-non-inversion}
 \sigma'(a) < \sigma'(b)\text{~ and ~}
 \sigma'\theta (a) > \sigma'\theta(b).
\end{equation}

If $\theta = (x,b)$, as $\theta \prec \tau$, then $\sigma(x) > \sigma(a)$. We also have $\sigma'(x) \geq \sigma(x)$ and $\sigma'(a)=\sigma(a)$ (no transposition on $a$ has yet been applied), so $\sigma'\theta(b) = \sigma'(x) > \sigma'(a)$ and \eqref{eq:proof-non-inversion} cannot be true. So $\theta = (a,y)$. We have $\theta \prec \tau = (a,b)$, so by Corollary \ref{cor:linear-extension}, we have $y>b$. As $\theta$ is a Bruhat transposition for $\sigma'$, we have either $\sigma'(b)<\sigma'(a)$  or $\sigma'(b) > \sigma'(y) = \sigma'\theta(a)$. Both relations contradict \eqref{eq:proof-non-inversion} and so such a transposition $\theta$ does not exist.
\end{proof}

\begin{proof}[Proof of Proposition \ref{prop:validSubwords}]
First, let us prove that if $w =  \tau_{i_1}  \dots \tau_{i_r}$ is a subword of $W_{\sigma,k}$ and is not a valid path, then $w$ is not compatible. Let $w_1$ be the longest left factor of $w$ which is a valid path. Then $w = w_1 \tau w_1'$ with $\tau=(b,d)$ not a valid Bruhat transposition for $\sigma w_1$. By Lemma \ref{lemme:non-inversion}, $\sigma w_1(b) < \sigma w_1(d)$ so there is $c$ such that $b<c<d$ and 

\begin{equation}
\label{eq:proof-valid-subword-conflit}
\begin{array}{ccccc}
 \sigma w_1(b)& < &\sigma w_1 (c)& < &\sigma w_1 (d)\\
 \vgeq        &   &              &   & \vleq         \\
 \sigma (b)   &   &              &   & \sigma(d)
\end{array}
\end{equation}

First, let us recall Theorem \ref{thm:KBruhatChar} and notice that $c>k$. Indeed, we have that $\tau$ is a Bruhat transposition for $\sigma$ so $\sigma(c)<\sigma(b)$ or $\sigma(c)>\sigma(d)$. In the first case, we read in \eqref{eq:proof-valid-subword-conflit} that $\sigma(c) < \sigma w_1 (c)$ so $c \leq k$. But as $\sigma(c) < \sigma(b)$, any transposition $(c,*)$ is after $\tau=(b,d)$ and then $\sigma w_1(c) = \sigma(c)$ which is not possible. Therefore, we have that $\sigma(c)>\sigma(d) \geq \sigma w_1 (d) > \sigma w_1 (c)$ which implies $c>k$. 

Let $(a,c)$ be the last transposition acting on position $c$. We know that $(a,c)$ exists because $\sigma(c) \neq \sigma w_1(c)$. We have $w_1 = w_2 (a,c) w_2'$. Also, as $\sigma(c) > \sigma(d)$ and $(a,c) \prec (b,d)$, then $a \neq b$ and we have 

\begin{equation}
\label{eq:proof-valid-subword-main}
\begin{array}{ccccc}
\sigma w_2 (d) & > & \sigma w_2 (a) & > & \sigma w_2 (b) \\
\vleq          &   & \vgeq          &   & \veq \\
\sigma(d) & & \sigma(a) & & \sigma(b)
\end{array}
\end{equation}

The vertical relations come from the $k$-Bruhat order between $\sigma$ and $\sigma w_2$ and the horizontal ones from~\eqref{eq:proof-valid-subword-conflit}. We claim that
\begin{align}
\label{eq:proof-valid-subword-values}
& \sigma(c) > \sigma(d) > \sigma (a) > \sigma (b) \text{~ and} \\
\label{eq:proof-valid-subword-positions}
& a < b < c < d.
\end{align}

We already have $\sigma(c)>\sigma(d)$ and we can read in \eqref{eq:proof-valid-subword-main} that $\sigma(d) > \sigma(a)$. Also $(a,c) \prec (b,d)$ gives $\sigma(a)>\sigma(b)$ and so \eqref{eq:proof-valid-subword-values} is true. To prove \eqref{eq:proof-valid-subword-positions}, we only have to show that $a<b$. It is done by saying that $(b,d)$ is a Bruhat transposition for $\sigma$, then because of \eqref{eq:proof-valid-subword-values}, we cannot have $b<a<d$.

By \eqref{eq:proof-valid-subword-values} and \eqref{eq:proof-valid-subword-positions}, we have that $(a,c)(b,d)$ is a subword of $w$ as defined in Definition \ref{def:compatible}, Condition \ref{enum:compatible1}. Now, $(a,d)$ is also a Bruhat transposition for $\sigma$. Indeed, if there is $a < x < d$ such that $\sigma(a) < \sigma(x) < \sigma(d)$, then either $x < c$ and $(a,c)$ is not a Bruhat transposition, or $x > c >b$ and $(b,d)$ is not a Bruhat transposition. So $(a,d) \in W_{\sigma, k}$ and because $\sigma w_2 (d) > \sigma(a)$, then $(a,d) \notin w$ and Condition \ref{enum:compatible2} of Definition \ref{def:compatible} is also satisfied, $w$ is a non compatible subword. 

Now, let $w$ be a non compatible subword of $W_{\sigma, k}$ from Definition \ref{def:compatible}, we prove that $w$ is not a valid path. By definition, $w$ contains at least one subword $(a,c)(b,d)$ satisfying conditions (\ref{enum:compatible1}) and (\ref{enum:compatible2}) of Definition \ref{def:compatible}. We chose one where the transposition $(a,c)$ is minimal. We write $w = w_1 (b,d) w_1'$, and we prove that if $w_1$ is a valid path, then $(b,d)$ is not a Bruhat transposition for $\sigma w_1$.

We have that $(a,c) \in w_1$, so $w_1 = w_2 (a,c) w_2'$. We show that 

\begin{equation}
\begin{array}{ccccc}
\sigma(b)       & < &\sigma(a)       & < & \sigma(d) \\
\veq            &   & \vleq          &   & \vgeq     \\
\sigma w_2 (b)  & < & \sigma w_2 (a) & < & \sigma w_2 (d).
\end{array}
\end{equation}

The relation $\sigma(b) < \sigma(a) < \sigma(d)$ is immediate as, by definition, $a<b<d$ and $(a,d),(b,d) \in W_{\sigma, k}$ which by Lemma \ref{prop:linear-extension} gives us $(a,d) \prec (b,d)$. Relations $\sigma(a) \leq \sigma w_2 (a)$ and $\sigma(d) \geq \sigma w_2(d)$ come from the $k$-Bruhat order. Any transposition acting on $b$ is such that $(b,*) \succ (a,c)$ so the value in position $b$ has not been changed. There only remains to prove $\sigma w_2 (a) < \sigma w_2 (d)$. Any transposition $\tau$ of $w_2$ is such that $\tau \prec (a,d)$. Indeed, if $(a,x) \in w$ with $(a,d) \prec (a,x) \prec (a,c)$, we have by Corollary~\ref{cor:linear-extension} that $c<x<d$ and then $(a,x)(b,d)$ satisfy conditions (\ref{enum:compatible1}) and (\ref{enum:compatible2}) of Definition \ref{def:compatible}. We have taken $(a,c)$ minimal, so such a transposition $(a,x)$ does not exist. Now we can apply Lemma \ref{lemme:non-inversion} and we have $\sigma w_2 (a) < \sigma w_2 (d)$. If we set $\sigma' = \sigma w_2 (a,c)$, we now have

\begin{equation}
\label{eq:proof-valid-subword-conflit2}
b<c<d \text{~ and ~} \sigma'(b) < \sigma'(c) < \sigma'(d).
\end{equation}

We claim that this relation is preserved through $w_2'$. If so, then $(b,d)$ is not a Bruhat transposition for $\sigma w_1$ and the result is proved. We show recursively that if $w_2' = w_3 \tau w_3'$ with $\sigma' := \sigma w_2 (a,c) w_3$ satisfying \eqref{eq:proof-valid-subword-conflit2}, then $\sigma' \tau$ still satisfy \eqref{eq:proof-valid-subword-conflit2}. 

We have $(a,c) \prec \tau \prec (b,d)$. If $\tau = (b,y)$, then by Corollary~\ref{cor:linear-extension}, $y>d>c$. As $\tau$ is a Bruhat transposition for $\sigma'$ and $\sigma'(c)>\sigma'(b)$ then $\sigma'(c)>\sigma'(y)$.  If $\tau = (y,c)$ or $(y,d)$, then $y<a<b$ and so $\sigma'(b)<\sigma'(y)$ because $\tau$ is a Bruhat transposition for $\sigma'$ and $\sigma'(b) < \sigma'(c) < \sigma'(d)$. In any case, $\sigma' \tau$ satisfy \eqref{eq:proof-valid-subword-conflit2}.
\end{proof}

Lemma \ref{lemme:intervalMax} is a direct consequence of Proposition \ref{prop:validSubwords}.

\begin{proof}[Proof of Lemma \ref{lemme:intervalMax}]
Let $\mu \in \mathfrak{E}_{\sigma, k}$, $\mu \neq \eta(\sigma, k)$. We prove that there exists $\tilde{\mu} \in \mathfrak{E}_{\sigma, k}$ with $\tilde{\mu}$ a direct successor of $\mu$. This is enough to prove Lemma \ref{lemme:intervalMax} as, if $\tilde{\mu} \neq \eta(\sigma, \mu)$, we can apply the algorithm recursively to find a direct successor of $\tilde{\mu}$ and so on. At each step, the length of the permutation is increased by one, so the process stops after $\ell(\eta) - \ell(\mu)$ iterations. 

For $\mu = \sigma w$ with $w$ a subword of $W_{\sigma, k}$ and $\mu \neq \eta$, let us set $\tau \in W_{\sigma, k}$, $\tau \notin w$ with $\tau$ minimal, \emph{i.e}, $\tau$ is the first transposition of $W_{\sigma,k}$ which is not in $w$. Then we define $\tilde{w} = u \tau v$, where $w = uv$ such that $\tilde{w}$ is still a subword of $W_{\sigma, k}$ ($\tau$ is greater than the transpositions of $u$ and smaller than the transpositions of $v$). Note that $u \tau$ is actually a left factor of $W_{\sigma, k}$ because $\tau$ is minimal among the transpositions that are not in $w$. We set $\tilde{\mu} = \sigma \tilde{w}$.

We prove by Proposition \ref{prop:validSubwords} that $\tilde{w}$ is a valid path. Let us suppose that $\tilde{w}$ is a noncompatible subword. Then we have $(a,d) \prec (a,c) \prec (b,d)$ all in $W_{\sigma, k}$ with $(a,d) \notin \tilde{w}$ and $(a,c),(b,d) \in \tilde{w}$. This is also true for $w$. Indeed, $\tilde{w} = u \tau v$ with $u \tau$ a left factor of $W_{\sigma, k}$. Which means that $(a,d) \notin u \tau$ implies $\tau \prec (a,d) \prec (a,c) \prec (b,d)$. The conditions on $(a,d)$, $(a,c)$ and $(b,d)$ are still satisfied by $w=uv$ which contradicts the fact that $w$ is a valid path. So $\tilde{w}$ is also a valid path. 

Then $\ell(\tilde{\mu}) = \ell(\sigma) + \vert \tilde{w} \vert = \ell(\sigma) + \vert w \vert +1 = \ell(\mu) + 1$ and $\tilde{\mu}$ = $\mu v^{-1} \tau v = \mu \theta$ where $\theta$ is the conjugate of a transposition, \textit{i.e.}, a transposition. So $\tilde{\mu}$ is a direct successor of $\mu$.
\end{proof}

Here is an example of the algorithm $\mu \rightarrow \tilde{\mu}$ described in the proof applied on $\mu = 1365 \vert 42$ for $\mathfrak{E}_{1362 \vert 54}$ drawn in Fig.~\ref{fig:ESigma}. 

\begin{align*}
\mu = \mu_0 &= (1362 \vert 54)(46)(45) =  1365 \vert 42 \\
\mu_1:= \tilde{\mu}_0 &= (1362 \vert 54)(26)(46)(45) = 1465 \vert 32 \\
\mu_2:= \tilde{\mu}_1 &=  (1362 \vert 54)(26)(25)(46)(45) = 1564 \vert 32 = \eta
\end{align*} 

\begin{Remarque}
\label{rqe:permutation-pattern}
If $W_{\sigma, k}$ contains $(a,d)(a,c)(b,d)$ as a subword, then we can find noncompatible subwords. We say that $W_{\sigma, k}$ contains a \textit{conflict pattern}. In particular, this implies that $\sigma$ contains the permutation pattern $21 \vert 43$ as we have $\sigma(b) < \sigma(a) < \sigma(d) < \sigma(c)$. But the converse implication is not true. As an example, the permutation $213 \vert 54$ contains the permutation pattern but as $(1,5)$,$(1,4)$, and $(2,5)$ are not Bruhat transpositions, $W_{\sigma, k}$ does not contain the conflict pattern. 
\end{Remarque}

\begin{Remarque}
\label{rqe:interval-size}
Proposition \ref{prop:validSubwords} also gives us information about the size of the interval. If $W_{\sigma,k}$ is of size $m$ and does not contain any conflict pattern, then $\vert \mathfrak{E}_{\sigma, k} \vert = 2^m$. Besides, the number of non compatible subwords due to a specific trio $(a,d) \prec (a,c) \prec (b,d)$ is $2^{m-3}$, so when $W_{\sigma, k}$ contains only one conflict pattern, we have $\vert \mathfrak{E}_\sigma \vert = 2^m - 2^{m-3}$. As an example, $\vert \mathfrak{E}_{1362 \vert 54} \vert = 2^4 - 2^1 = 14$. More generally, we have to perform an inclusion-exclusion algorithm. If the number of conflict patterns is high, the inclusion-exclusion algorithm can take much longer to compute than the actual generation of $\mathfrak{E}_{\sigma, k}$. The permutation $4321 \vert 8765$ contains 36 conflict patterns and one should compute billions of intersections of conflict patterns whereas the size of $\mathfrak{E}_\sigma$ is only 6902. But the algorithm can be used efficiently on large permutations with few conflict patterns. 
\end{Remarque}

\section{Other parabolic subgroups}
\label{sec:parabolic-subgroups}

When we compute $\mathfrak{E}_{\zeta, k} \pi_{\zeta^{-1} \sigma}$, we apply operators $\pi_i$ where $i \neq k$, \emph{i.e.}, an element of the parabolic subgroup generated by $ \lbrace \pi_1, \dots, \pi_{k-1}, \pi_{k+1}, \allowbreak \dots, \pi_n \rbrace$. Proposition \ref{prop:ESigma} can be generalized to other parabolic subgroups. 

\begin{Definition}
\label{def:gen-blocks}
Let $\sigma \in \mathfrak{S}_n$ and $1 \leq k_1 < k_2 < \dots < k_m < n$. By convention, we set $k_0 = 0$. Then the block $\beta_i(\sigma)$ for $0 < i \leq m$ is the word $\left[ \sigma(k_{i-1}+1), \sigma(k_{i-1}+2, \dots, \sigma(n) \right]$ and $ \mathcal{W}_i$ is the list of transpositions 

\begin{align}
\nonumber
(&(k_{i-1} +1, n), (k_{i-1}+1, n-1), \dots, (k_{i-1}+1, k_i+1),\\
\nonumber
 &(k_{i-1} +2, n), (k_{i-1}+2, n-1), \dots, (k_{i-1}+2, k_i+1),\\
\nonumber
 &\dots,\\
 &(k_i, n), (k_i, n-1), \dots, (k_i, k_i+1)).
\end{align}
We define $ \mathcal{W}_{(k_1, \dots, k_m)}$ to be the concatenation of the lists $ \mathcal{W}_{i}$ for $m \geq i \geq 1$ \emph{i.e.}, $ \mathcal{W}_{(k_1, \dots, k_m)} := ( \mathcal{W}_{m}, \dots, \mathcal{W}_{1} )$.
\end{Definition}

As an example, if $\sigma = 43 \vert 2 \vert 836 \vert 57$ ($k_1 = 2$, $k_2 = 3$, $k_3 = 6$), then $\beta_1 =  43 \vert 283657$, $\beta_2 = 2 \vert 83657$, $\beta_3 = 836 \vert 57$ and $ \mathcal{W}_3 = ( (4,8),(4,7), (5,8), \allowbreak (5,7),(6,8),(6,7))$, $\mathcal{W}_2 = ((3,8),(3,7),(3,6),(3,5),(3,4))$, $\mathcal{W}_1 = ((1,8), \allowbreak (1,7),(1,6),(1,5),(1,4),(1,3),(2,8),(2,7),(2,6),(2,5),(2,4),\allowbreak(2,3))$.

\begin{Proposition}
\label{thm:gen-esigma}
For $\sigma \in \mathfrak{S}_n$ and $1 \leq k_1 < k_2 < \dots < k_m < n$, let $\zeta = \zeta(\sigma,k_1, \dots, k_m)$ be the element of the coset $\sigma (\mathfrak{S}_{k_1} \times \mathfrak{S}_{k_2 - k_1} \times \mathfrak{S}_{k_3 - k_2} \times \dots \times \mathfrak{S}_{n - k_m})$ with maximal length. Then
\begin{equation}
\label{eq:gen-esigma}
K_{\omega} \hat{\pi}_{\omega \zeta} \pi_{\zeta^{-1} \sigma} = \mathfrak{E}_{\sigma, (k_1, \dots, k_m)}
\end{equation}
where 
\begin{equation}
\mathfrak{E}_{\sigma, (k_1, \dots, k_m)}  = \sum_{w} (-1)^{\vert w \vert} K_{\sigma w}
\end{equation}
summing over subwords of the list $\mathcal{W}_{(k_1, \dots, k_m)}$ that correspond to valid chains of the Bruhat order starting at $\sigma$.

Besides, the sum is cancellation free, the coefficients are $\pm 1$, and the summing set is closed by interval.
\end{Proposition}

Note that when $m=1$, the enumeration is the one described in Proposition \ref{prop:ESigma}, in this case the summing set has a unique maximal element. This is not true in general when $m>1$. As an example, for $\sigma = 25 \vert 14 \vert 63$, $k_1 = 2$, $k_2 = 4$, one obtain two elements with maximal length: $362541 = 251463 (3,6)(4,5)(1,3)(2,4)$ and $461532 = 251463 (4,5)(1,6)(1,5)(2,4)$.

First, the fact that $\mathfrak{E}_{\sigma, k_1, \dots, k_m}$ is closed by interval can easily be proved by a generalization of Proposition \ref{prop:complement}. 

\begin{Proposition}
\label{prop:gen-complement}
Let $\Succs_{\sigma}^{-(k_1, \dots, k_m)}$ be the set of successors of $\sigma$, $\sigma' = \sigma \tau$ where $\tau$ is not a $k_i$-transposition of $\sigma$ for $1 \leq i \leq m$. Then we have:

\begin{equation}
K_{\omega} \hat{\pi}_{\omega \zeta} \pi_{\zeta^{-1} \sigma}  = \sum (-1)^{\ell(\nu) - \ell(\sigma)} K_\nu
\end{equation}
summing over permutations $\nu$ such that $\nu \geq \sigma$ and, $\forall \sigma' \in \Succs_{\sigma}^{-(k_1, \dots, k_m)}$, $\nu \ngeq \sigma'$. 
In particular, the set of permutations of the summation is closed by interval.
\end{Proposition}

\begin{proof}
The proof is completely similar to that of Proposition \ref{prop:complement}. One just has to see that Lemma \ref{lemme:coset} can be generalized to cosets of the form $\sigma (\mathfrak{S}_{k_1} \times \mathfrak{S}_{k_2 - k_1} \times \mathfrak{S}_{k_3 - k_2} \times \dots \times \mathfrak{S}_{n - k_m})$. This can be done by induction as we have $\sigma (\mathfrak{S}_{k_1} \times \mathfrak{S}_{n-k_1}) \supset \sigma(\mathfrak{S}_{k_1} \times \mathfrak{S}_{k_2 - k_1} \times \mathfrak{S}_{n-k_2}) \supset \dots \supset \sigma (\mathfrak{S}_{k_1} \times \mathfrak{S}_{k_2 - k_1} \times \dots \times \mathfrak{S}_{n - k_m})$.
\end{proof}

\begin{proof}[Proof of Proposition \ref{thm:gen-esigma}]
The proof is done by induction on $m$. When $m = 1$, it corresponds to Proposition \ref{prop:ESigma}. Now, let $\zeta'$ be the element of $\sigma (\mathfrak{S}_{k_1} \times \mathfrak{S}_{k_2 - k_1} \times \mathfrak{S}_{k_3 - k_2} \times \dots \times \mathfrak{S}_{n - k_{m-1}})$ with maximal length. The inversions of $\zeta'$ contain the inversions of $\zeta$ and so, $\zeta' \geq \zeta$ for the right weak order. This means that any reduced decomposition of $\omega \zeta'$ is a prefix of a reduced decomposition of $\omega \zeta$ and we have:

\begin{equation}
K_{\omega} \hat{\pi}_{\omega \zeta} \pi_{\zeta^{-1} \sigma} = K_{\omega} \hat{\pi}_{\omega \zeta'} \hat{\pi}_{\zeta'^{-1} \zeta} \pi_{\zeta^{-1} \sigma}.
\end{equation}

All operators $\hat{\pi}_i$ in $\hat{\pi}_{\zeta'^{-1} \zeta}$ are such that $i>k_{m-1}$. Also, there is no operator $\pi_{k_{m-1}}$ in the product $\pi_{\zeta^{-1} \sigma}$, so operators $\pi_i$ with $i<k_{m-1}$ commute with operators  $\pi_i$ with $i>k_{m-1}$ and with operators of $\hat{\pi}_{\zeta'^{-1} \zeta}$ and we can write

\begin{equation}
K_{\omega} \hat{\pi}_{\omega \zeta} \pi_{\zeta^{-1} \sigma} = (K_{\omega} \hat{\pi}_{\omega \zeta'}  \pi_{\zeta^{-1} \sigma}^{(<k_{m-1})})(\hat{\pi}_{\zeta'^{-1} \zeta} \pi_{\zeta^{-1} \sigma}^{(>k_{m-1})}).
\end{equation}

We have $\pi_{\zeta^{-1} \sigma}^{(<k_{m-1})} = \pi_{\zeta'^{-1} \tilde{\sigma}}$ where $\tilde{\sigma} = [\sigma(1), \sigma(2),\allowbreak \dots, \sigma(k_{m-1}), \zeta'(k_{m-1} +1), \zeta'(k_{m-1}+2), \dots, \zeta'(n)]$. The product $(\hat{\pi}_{\zeta'^{-1} \zeta} \pi_{\zeta^{-1} \sigma}^{(>k_{m-1})})$ only acts on the block $\beta_m(\sigma)$, one can use Proposition \ref{prop:ESigma} by ignoring the left factor of size $k_{m-1}$ of $\sigma$. We have

\begin{equation}
K_{\omega} \hat{\pi}_{\omega \zeta} \pi_{\zeta^{-1} \sigma} = \sum_{u} (-1)^{\vert u \vert} K_\omega \hat{\pi}_{\omega \zeta'} \pi_{\zeta'^{-1} \tilde{\sigma}} \pi_{\tilde{\sigma}^{-1} \sigma u}
\end{equation}
summing over subwords $u$ of $\mathcal{W}_m$ that are valid chains in the Bruhat order starting at $\sigma$. Now, as $\pi_{\zeta'^{-1} \tilde{\sigma}}$ and $ \pi_{\tilde{\sigma}^{-1} \sigma u}$ contain only operators $\pi_i$ with respectively $i<k_{m-1}$ and $i>k_{m-1}$, the product is still reduced and
\begin{equation}
K_{\omega} \hat{\pi}_{\omega \zeta} \pi_{\zeta^{-1} \sigma} = \sum_{u} (-1)^{\vert u \vert} K_\omega \hat{\pi}_{\omega \zeta'} \pi_{\zeta'^{-1} \sigma u}.
\end{equation}
We now obtain \eqref{eq:gen-esigma} by induction. 

By Proposition \ref{prop:gen-complement}, we know that the coefficients of the different permutations are either $1$ or $-1$ depending on the length. This tells us that the sum is cancellation free and that each permutation is obtained in exactly one chain.  
\end{proof}

As an example, let us see the computation for $\sigma = 12 \vert 463 \vert 5$, $k_1 = 2$, $k_2 = 5$. We have $\zeta = 21 \vert 643 \vert 5$ and $\zeta' = 21 \vert 6543$. The result is illustrated by Fig. \ref{fig:ESigmaMultiK}.

\begin{align}
K_{\omega} \hat{\pi}_{\omega \zeta} \pi_{\zeta^{-1} \sigma}  &= K_{654321} \hat{\pi}_4  \hat{\pi}_3  \hat{\pi}_2  \hat{\pi}_1 \hat{\pi}_5 \hat{\pi}_4 \hat{\pi}_3 \hat{\pi}_2 \hat{\pi}_4  \hat{\pi}_5~ \pi_1 \pi_3 \\
&= (K_{654321} \hat{\pi}_4  \hat{\pi}_3  \hat{\pi}_2  \hat{\pi}_1 \hat{\pi}_5 \hat{\pi}_4 \hat{\pi}_3 \hat{\pi}_2 ~ \pi_1)  (\hat{\pi}_4  \hat{\pi}_5 ~ \pi_3) \\
&= \hat{K}_{216543} \pi_1 (\pi_4 \pi_5 \pi_3 - \pi_4\pi_3 - \pi_5\pi_3 + \pi_3) \\
&= \mathfrak{E}_{12 \vert 4635} - \mathfrak{E}_{12 \vert 4653} - \mathfrak{E}_{12 \vert 5634} + \mathfrak{E}_{12 \vert 5643}
\end{align}

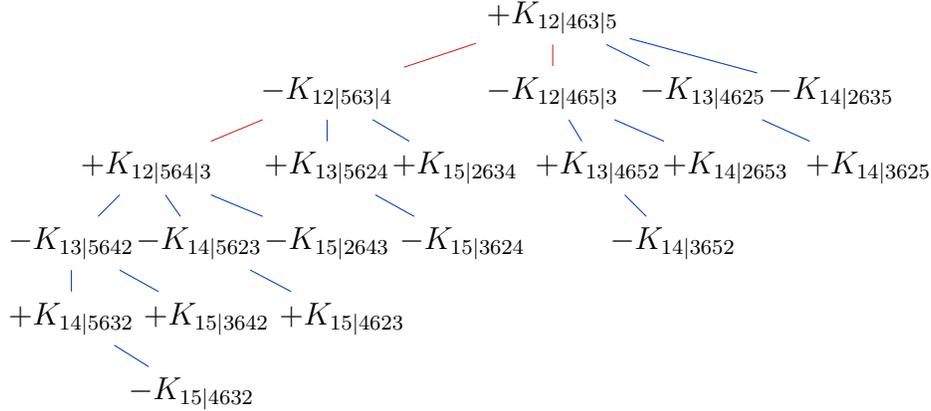
\begin{figure}[ht]
\centering
\begin{tikzpicture}
\node(K13524) at (6.4,0){$+K_{12 \vert 463 \vert 5}$};

\node(K14523) at (3.4,-1){$- K_{12 \vert 563 \vert 4}$};
\node(K13542) at (6.4,-1){$- K_{12 \vert 465 \vert 3}$};
\node(K23514) at (8.4,-1){$- K_{13 \vert 4625}$};
\node(K31524) at (10.1,-1){$- K_{14 \vert 2635}$};

\node(K14532) at (1,-2){$+ K_{12 \vert 564 \vert 3}$};
\node(K24513) at (3.4,-2){$+ K_{13 \vert 5624}$};
\node(K41523) at (5.1,-2){$+ K_{15 \vert 2634}$};
\node(K23541) at (7,-2){$+ K_{13 \vert 4652}$};
\node(K31542) at (8.7,-2){$+ K_{14 \vert 2653}$};
\node(K32514) at (10.6,-2){$+ K_{14 \vert 3625}$};

\node(K24531) at (0,-3){$- K_{13 \vert 5642}$};
\node(K34512) at (1.7,-3){$- K_{14 \vert 5623}$};
\node(K41532) at (3.4,-3){$- K_{15 \vert 2643}$};
\node(K42513) at (5.2,-3){$- K_{15 \vert 3624}$};
\node(K32541) at (8,-3){$- K_{14 \vert 3652}$};

\node(K34521) at (0,-4){$+ K_{14 \vert 5632}$};
\node(K42531) at (1.8,-4){$+ K_{15 \vert 3642}$};
\node(K43512) at (3.6,-4){$+ K_{15 \vert 4623}$};

\node(K43521) at (1.6,-5){$- K_{15 \vert 4632}$};

\draw[color=Rouge] (K13524) -- (K14523);
\draw[color=Rouge] (K13524) -- (K13542);
\draw[color=Rouge] (K14523) -- (K14532);

\draw[color=Bleu] (K14532) -- (K24531);
\draw[color=Bleu] (K14532) -- (K34512);
\draw[color=Bleu] (K14532) -- (K41532);
\draw[color=Bleu] (K24531) -- (K34521);
\draw[color=Bleu] (K24531) -- (K42531);
\draw[color=Bleu] (K34512) -- (K43512);
\draw[color=Bleu] (K34521) -- (K43521);

\draw[color=Bleu] (K14523) -- (K24513);
\draw[color=Bleu] (K14523) -- (K41523);
\draw[color=Bleu] (K24513) -- (K42513);

\draw[color=Bleu] (K13542) -- (K23541);
\draw[color=Bleu] (K13542) -- (K31542);
\draw[color=Bleu] (K23541) -- (K32541);

\draw[color=Bleu] (K13524) -- (K23514);
\draw[color=Bleu] (K13524) -- (K31524);
\draw[color=Bleu] (K23514) -- (K32514);

\end{tikzpicture}
\caption{
The set $\mathfrak{E}_{12 \vert 463 \vert 5}$.
}
\label{fig:ESigmaMultiK}
\end{figure}

\subsubsection*{Acknowledgements}
\label{sec:ack}
The computation and tests needed along the research were done using the open-source mathematical softwear \texttt{Sage} \cite{sage} and its combinatorics features developed by the \texttt{Sage-Combinat} community \cite{Sage-Combinat}.

\bibliographystyle{plain}
\label{sec:biblio}
\bibliography{article_postprint}
\end{document}